\documentclass[hidelinks,12pt,a4paper]{article}
\usepackage[top=30mm, bottom=30mm, inner=20mm, outer=20mm, headsep=10mm, footskip=12mm]{geometry}
\usepackage{graphicx} 
\usepackage{amsmath}
\usepackage{mathtools}
\usepackage{mathrsfs}
\usepackage{amsthm}
\usepackage{amssymb}
\usepackage{graphics}
\usepackage{longtable}
\usepackage[english]{babel}
\usepackage{tikz}
\usepackage{esint}

\newtheorem{Theorem}{Theorem}[section]

\newtheorem{Remark}[Theorem]{Remark}
\newtheorem{Example}[Theorem]{Example}
\newtheorem{Proposition}[Theorem]{Proposition}
\newtheorem{Lemma}[Theorem]{Lemma}

\newcommand{\R}{\mathbb R}
\newcommand{\Z}{\mathbb{Z}}
\newcommand{\eps}{\varepsilon}
\let\pphi\phi
\let\phi\varphi

\newcommand\blfootnote[1]{%
  \begingroup
  \renewcommand\thefootnote{}\footnote{#1}%
  \addtocounter{footnote}{-1}%
  \endgroup
}

\title{Homogenization of non-local energies\\ on disconnected sets}
\author{Andrea Braides, Sergio Scalabrino, and Chiara Trifone \\
\small SISSA, via Bonomea 265, Trieste, Italy}
\date{}
\begin{document}
\maketitle

\blfootnote{Preprint SISSA  09/2024/MATE}

%\hfill{\em Dedicated to the memory of Hedy Attouch}

\def\e{\varepsilon}

\begin{abstract}
    We consider the problem of the homogenization of non-local quadratic energies defined on $\delta$-periodic disconnected sets defined by a double integral, depending on a kernel concentrated at scale $\e$. For kernels with unbounded support we show that we may have three regimes: (i) $\e<\!<\delta$, for which the $\Gamma$-limit even in the strong topology of $L^2$ is $0$; (ii) $\frac\e\delta\to\kappa$, in which the energies are coercive with respect to a convergence of interpolated functions, and the limit is governed by a non-local homogenization formula parameterized by $\kappa$; (iii) $\delta<\!<\e$, for which the $\Gamma$-limit is computed with respect to a coarse-grained convergence and exhibits a separation-of-scales effect; namely, it is the same as the one obtained by formally first letting $\delta\to 0$ (which turns out to be a pointwise weak limit, thanks to an iterated use of Jensen's inequality), and then, noting that the outcome is a nonlocal energy studied by Bourgain, Brezis and Mironescu, letting  $\e\to0$. A slightly more complex description is necessary for case (ii) if the kernel is compactly supported.  


    {\bf MSC Class (2020): 35B27, 74Q05, 49J45, 26A33, 74A70}
\end{abstract}

\section{Introduction}
In this paper we study the asymptotic behaviour of a periodic system of disconnected regions that interact through long-range potentials and can therefore be considered `energetically' connected in the sense of V.V.~Zhikov's {\em $p$-connectedness} \cite{zhikov1996connectedness}. The study of such a type of geometrical objects falls within the analysis of those commonly referred to as `perforated domains', where the domain of integration is a portion of a scaled periodic set $E$ contained in a given region; that is, a set of the form $\Omega\cap \delta E$, with $\delta>0$ a small parameter. The `classical' case is obtained by choosing as energy e.g.~the Dirichlet 
integral, for which the functionals take the form
\begin{equation}
F_\delta(u)=\int_{\Omega\cap \delta E}|\nabla u|^2\,dx.
\end{equation}
If $E$ is a connected open set, then the energies possess a limit, e.g.~in the sense of $\Gamma$-convergence \cite{Attouch,DM,GCB}, as $\delta\to 0$ \cite{BDF} and it is a coercive homogeneous quadratic form on $H^1(\Omega)$. This result can be obtained by using Extension Lemmas, which allow to construct uniformly continuous operators from $H^1(\Omega\cap \delta E)$ to $H^1_{\rm loc}(\Omega)$ \cite{ACPDMP}, and then regard all the functionals as defined in that common space. For less regular sets it is convenient to substitute the topological notion of connectedness with a more analytic one: loosely speaking, if $p>1$, a set $A$ is {\em $p$-connected} if every function $u$ such that $\int_A|\nabla u|^pdx=0$ is a constant \cite{zhikov1996connectedness}. More in general this notion can be given for integrals with respect to a measure, of which the restriction of the Lebesgue measure to $A$ is a particular case. 

Following a seminal result of Bourgain et al.~\cite{bourgainbrezis} (see also \cite{ponce}, and \cite{MR3396427,MR3424902} for applications) the Dirichlet integral on an open set $U\subset \mathbb R^d$ can itself be approximated by introducing another parameter $\e>0$ and considering energies
 \begin{equation}
F^\e(u)=\frac1{\e^{d}}\int_{U\times U}\varphi\big(\tfrac{x-y}\e\big)\frac{|u(x)-u(y)|^2}{|x-y|^2}\,dx\,dy,
\end{equation}
with $\varphi$ a non-negative integrable kernel. 
%For simplicity we will treat the case that $\varphi$ is radially symmetric and decreasing from the origin.
Such energies have also been systematically studied within a variational theory for convolution-type functionals \cite{alicandro2023variational} in the equivalent form
 \begin{equation}
F^\e(u)=\frac1{\e^{d+2}}\int_{U\times U}\varphi\big(\tfrac{x-y}\e\big)|u(x)-u(y)|^2\,dx\,dy,
\end{equation}
up to using the kernel $\varphi(\xi)|\xi|^2$ in the place of  $\varphi(\xi)$.
We will use this latter form.
Note that if the support of $\varphi$ is  unbounded then any set $U$ is in a sense `$2$-connected' for such functionals, in the sense that if $F^\e(u)=0$ then $u(x)-u(y)=0$ for each $x,y\in U$ so that $u$ is constant on $U$. 

This observation suggests that, contrary to the local case, for non-local energies we might study the behaviour of perforated domains as above even when $E$ is not topologically connected.
In this paper, we consider non-local functionals of the form
\begin{equation}\label{intro-6}
F_{\e,\delta}(u)=\frac1{\e^{d+2}}\int_{\Omega\cap \delta E}\int_{\Omega\cap \delta E}\varphi\big(\tfrac{x-y}\e\big)|u(x)-u(y)|^2\,dx\,dy,
\end{equation}
under the prototypical assumption that 
$$
E= K+\mathbb Z^d
$$
is composed of a periodic array of disconnected sets; that is, $K$ is a (topologically) connected compact set and $(K+k)\cap (K+k')=\emptyset$ if $k,k'\in\mathbb Z^d$ and $k\neq k'$. The behaviour of $F_{\e,\delta}$ will be driven by the mutual behaviour of the two parameters. In the notation, we will tacitly suppose that $\delta=\delta_\e$ is an infinitesimal family as $\e\to 0$. Functionals of the form \eqref{intro-6} have been dealt with in \cite{braides2019homogenization} (see also \cite{alicandro2023variational}) when $E$ is a topologically connected periodic Lipschitz set, using an extension theorem which cannot be applied in our cases. 

\bigskip
The first issue that we examine for such energies is their coerciveness. Since the domain of $F_{\e,\delta}$ is composed of functions defined on a disconnected set, we have to specify the type of convergence with respect to which they are studied. 
In the case of $\varphi$ with support the whole $\mathbb R^d$ we obtain the three cases:

\smallskip
(i) $\e<\!<\delta$. In this case we have a loss of coerciveness, and in particular any function in $L^2(\Omega)$ is a limit in $L^2(\Omega)$ of a sequence of functions with vanishing energy;

\smallskip
(ii) $\e\sim\delta$. In this case we consider the convergence $u_\e\to u$ defined as the $L^2_{\rm loc} (\Omega)$ strong convergence of the piecewise-constant functions $u^\e$ defined by 
\begin{equation}\label{def convergence delta sim eps}
u^\e(x)=\frac1{|K|\delta^{d}}\int_{\delta K+\delta k}u_\e(y)\,dy\hbox{ for } x\in \delta( k+ [0,1)^d).
\end{equation}
The functionals are equicoercive with respect to this convergence, and the limit $u$ belongs to $H^1(\Omega)$;

(iii) $\delta<\!<\e$. In this case the functions $u_\e$ must be `coarse grained', considering
the strong $L^2_{\rm loc} (\Omega)$-convergence of the piecewise-constant interpolations $u^\e$ defined by 
\begin{equation}\label{def convergence delta ll eps}
u^\e(x)=\frac1{|[0,\e]^d\cap \delta E |}\int_{\e( k+ [0,1)^d)\cap \delta E}u_\e(x)\,\hbox{ for } x\in \e( k+ [0,1)^d).
\end{equation}
Also in this case, the limit $u$ is in $H^1(\Omega)$.

\smallskip
If the support of $\varphi$ is bounded then we may have loss of coerciveness also if $\e\sim\delta$. More precisely, if the support of $\varphi$ is the closed ball of radius $s$, then the functionals are equicoercive with respect to the convergence above if and only if $s\e>\delta D$, where $D=\inf\{r: E+B_{r/2} \hbox{ is topologically connected}\}$. In this case again the limit $u$ belongs to $H^1(\Omega)$.
\bigskip

We also compute the $\Gamma$-limit with respect to the convergences above.
If the support of $\varphi$ is the whole $\mathbb R^d$ we have:
\smallskip

(i) ({\em degenerate limit}) if $\e<\!<\delta$ then the $\Gamma$-limit with respect to the strong $L^2(\Omega)$-convergence is identically $0$;

\smallskip
(ii) ({\em homogenization}) if $\frac\e\delta\to \kappa$, then the $\Gamma$-limit is the quadratic form 
$$
F^\kappa(u)=\int_\Omega \langle A^\kappa_{\rm hom} \nabla u, \nabla u\rangle\,dx,
$$
where the symmetric matrix $A^\kappa_{\rm hom}$ satisfies
$$
\langle A^\kappa_{\rm hom}\xi,\xi\rangle=
\min\Bigl\{\int_{[0,1]^d\cap E}\int_{E} \varphi(\tfrac{x-y}\kappa)(\langle\xi, x-y\rangle+ u(x)-u(y))^2\,dx\,dy:  u \hbox{ $1$-periodic}\Big\}.
$$
This result can be seen as following from the results in \cite{braides2019homogenization} (see also \cite{alicandro2023variational}) using the techniques therein combined with the coerciveness result above;

\smallskip
(iii) ({\em separation of scales}) if $\delta<\!<\e$ and $\varphi$ is radially symmetric, then the $\Gamma$-limit is given by
$$
F^\infty(u)=C_\infty\int_\Omega |\nabla u|^2dx,
\ \hbox{  where }\ 
C_\infty= |K|^2\frac1d\int_{\mathbb R^d} \varphi(\xi)|\xi|^2d\xi.
$$
Note that $\frac1d\int_{\mathbb R^d} \varphi(\xi)|\xi|^2d\xi$ is the constant appearing in the $\Gamma$-limit by Bourgain et al.~\cite{bourgainbrezis}, so that this limit can be obtained by first letting $\delta\to 0$, noting that, upon writing
$$
F_{\e,\delta}(u)=\frac1{\e^{d+2}}\int_{\Omega\times \Omega} \chi_{E\times E}(\tfrac{x}\delta,\tfrac{y}\delta)\varphi\big(\tfrac{x-y}\e\big){|u(x)-u(y)|^2}\,dx\,dy,
$$
the corresponding $\Gamma$-limit is simply 
$$
F^\infty_{\e}(u)=\frac1{\e^{d+2}}|K|^2\int_{\Omega\times \Omega} \varphi\big(\tfrac{x-y}\e\big)|u(x)-u(y)|^2\,dx\,dy,
$$
whose successive $\Gamma$-limit as $\e\to 0$ is $F^\infty$. This result is obtained using the Fonseca M\"uller blow-up method combined by a convexity argument for the lower bound. These arguments allow first to reduce to test functions which are $\delta$-periodic perturbations of affine functions, and then to obtain the desired inequality by a double Jensen's inequality. The upper bound is obtained by a direct computation when the target function $u$ is $C^2$. In this case we can take $u_\e=u$, in which case a discretization argument allows to write $F_{\e,\delta}(u)$ as an approximation of a Riemann integral. As a technical remark, we note that it is sufficient to treat the case $\varphi=\chi_{B_r}$ since a general $\varphi$ can be approximated by linear combinations of this type of energies. The lower bound then follows by the superadditivity of the lower limit, while the 
upper bound is proved by the pointwise convergence on $C^2$-functions.

Finally, for $\varphi$ with support the closed ball of radius $s$ the computation in (ii) also holds provided $s\kappa>D$, so that the functionals are equi-coercive.

\section{Notation and statement of the results}\label{section: notation and statement of the results}
We consider a radial convolution kernel in $\mathbb R^d$; that is, a function $\varphi:\mathbb R^d\to [0,+\infty)$ such that a decreasing function $\phi_0:[0,+\infty)\to  [0,+\infty)$ exists satisfying $\varphi(\xi)=\phi_0(|\xi|)$. We further assume that
\begin{equation}\label{crescfi}
\int_{\mathbb R^d} \varphi(\xi)(1+|\xi|^2)dx<+\infty,
\end{equation}
and for each $\e>0$ we define the scaled kernel $\varphi_\e$ by
$$
\varphi_\e(\xi)= \tfrac1{\e^d}\varphi(\tfrac\xi\e).
$$ 
A simple kernel, which will be used as a comparison for general kernels, is 
    $$
    \phi_{\eps}(\xi)=
    \begin{cases}
        \eps^{-d} & \text{if } |\xi| < \eps,\\
        0 & \text{elsewhere},
    \end{cases} 
    $$
    obtained with $\varphi_0=\chi_{[0,1]}$.

\smallskip
We consider a periodic set $E\subset \mathbb R^d$, and we fix a bounded domain $\Omega \subset \R^d$ with Lipschitz boundary.
For all $\e>0$ and $\delta>0$ we define the functional $F_{\delta,\eps}:L^2(\Omega) \rightarrow [0,+\infty) $ by
    \begin{equation}\label{funzionale generale}
    F_{\delta,\eps}(u) = \frac1{\e^2}\int_{(\Omega \cap \delta E) \times (\Omega \cap \delta E)} \phi_{\eps}(x-y)|u(x)-u(y)|^2dxdy.
    \end{equation}
Note that if $E=\mathbb R^d$ then $ F_{\delta,\eps}$ is independent of $\delta$ and the $\Gamma$-limit with respect to the weak and strong convergence in $L^2(\Omega)$ of $F_\e(u)= F_{\delta,\eps}(u)$ has been shown to be equal to
\begin{equation}\label{ci-fi}
C_\varphi \int_\Omega|\nabla u|^2dx, \hbox{ where } C_\varphi:= \frac1d\int_{\mathbb R^d} \varphi(\xi)|\xi|^2dx.     
\end{equation}

%Consider a smooth bounded domain $\Omega \subset \R^d$. We define the functional $F_{\delta,\eps}:L^2(\Omega) \rightarrow \R $ by
%    \begin{equation}\label{funzionale generale}
%    F_{\delta,\eps}(u) = \iint_{(\Omega \cap \delta E) \times (\Omega \cap \delta E)} \phi_{\eps}(x-y)\frac{|u(x)-u(y)|^2}{|x-y|^2}dxdy
%    \end{equation}
%where $E=B_c+\Z^d$, with $c<\frac{1}{2}$, so $\delta E$ is the set of balls of radius $c\delta$ centered in $\delta k$, for $k \in \Z^d$.  The function $\phi_{\eps}$ is a usual convolution-type kernel, which is radial and localized in a ball of radius $\eps$, such as 
%    $$
%    \phi_{\eps}(\xi)=
%    \begin{cases}
%        \eps^{-d} & \text{if } |\xi| < \eps,\\
%        0 & \text{elsewhere}.
%    \end{cases} 
%    $$
%    
    We will consider a set $E$ composed of disconnected components; more precisely,
a $1$-periodic set $E$ in $\mathbb R^d$ of the form 
\begin{equation}
E=\sum_{k\in\mathbb Z^d} (k+K),
\end{equation}
where $K$ is the closure of a connected open set with boundary of zero measure, and is such that $(k+K)\cap K=\emptyset$ if $k\in \mathbb Z^d$ and $k\neq 0$.
We also define 
\begin{equation}\label{Di}
D=\inf\{r: \, E+B_{r/2}\hbox{ is (topologically) connected}\}.
\end{equation}

\begin{figure}[ht]
\centering
\begin{tikzpicture}[domain=0:4]
  \draw[step=2,very thin,color=gray] (-2.1,-1.1) grid (4.9,4.9);

  \draw[->] (-1.2,0) -- (5.2,0) node[right] {$x$};
  \draw[->] (0,-1.2) -- (0,5.2) node[above] {$y$};
  \draw[solid,fill=gray, fill opacity = 0.2] (0,0) circle (0.5);
  \draw[solid,fill=gray, fill opacity = 0.2] (2,0) circle (0.5);
  \draw[solid,fill=gray, fill opacity = 0.2] (0,2) circle (0.5);
  \draw[solid,fill=gray, fill opacity = 0.2] (2,2) circle (0.5);
  \draw[->]        (1,-0.7)   -- (1.95,-0.7);
  \draw[->]        (1,-0.7)   -- (0.05,-0.7);
  \node[text width=1cm] at (1.4,-0.5){$\delta$};
  \draw[dashed, ->,line width=0.1mm]      (-1,0.7)   -- (-0.05,0.2); 
  \draw[red]            (0,0) -- (0,0.5);
  \node[text width=1cm] at (-1,0.7){$c\delta$};
  %\draw[color=red]    plot (\x,\x)             node[right] {$f(x) =x$};
  % \x r means to convert '\x' from degrees to _r_adians:
  %\draw[color=blue]   plot (\x,{sin(\x r)})    node[right] {$f(x) = \sin x$};
  %\draw[color=orange] plot (\x,{0.05*exp(\x)}) node[right] {$f(x) = \frac{1}{20} \mathrm e^x$};
  \draw [blue,dashed,domain=-30:120] plot ({2.3*cos(\x)}, {2.3*sin(\x)});
  \draw[blue, dashed, ->] (0,0) -- (1.55,1.55);
  \node[text width=1cm] at (1.1,1.1){$\eps$};
  \node[text width=1cm] at (2.7,2.7){$\delta E$};
  \node[text width=1cm] at (1.4,3.3){$\Omega$};
\end{tikzpicture}
\caption{ a simple example of non-connected domain}\label{Fig1}
\end{figure}
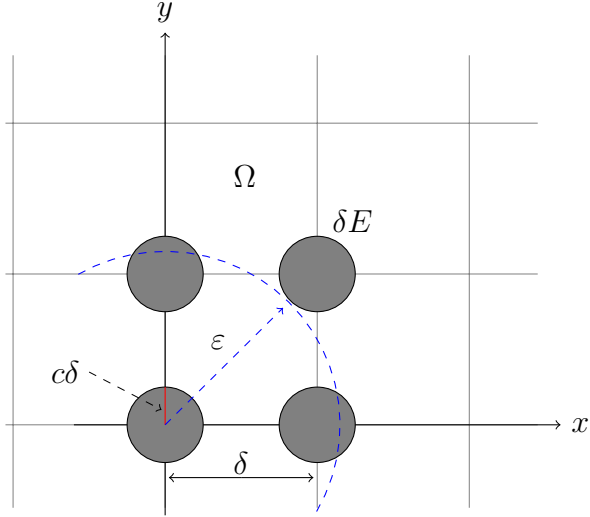

\begin{Example}\label{Ex1}\rm The simplest example of such a geometry is by taking $K=\overline{B_c}$ the closed ball of center $0$ and radius $c<\frac12$ (see Fig~\ref{Fig1}), for which $D=1-2c$. Note that for this choice of $E$ ${\rm dist}(K, K+k)\ge D$ for all $k\in\mathbb R^d$, which is not the case in general. For example, if we take as $K$ the closed ellipse given by the relation $\frac{x_1^2}{c^2_1}+\frac{x_2^2}{c^2_2}\le 1$ with $c_2<c_1$
in $\mathbb R^2$, we have dist$(K, K+e_1)= 1-2c_1$ and  dist$(K, K+e_2)= 1-2c_2=D$.\end{Example}

\subsection{Definition of convergence and coerciveness}

We first consider the cases in which we do not have coerciveness.
The first one is in the regime $\e<\!<\delta$, for which the $\Gamma$-limit, even if computed with respect to the strong $L^2(\Omega)$ convergence is $0$. This is a consequence of the following result.

\begin{Theorem}\label{no-comp-1} Let $\delta=\delta_\e$ be such that $\delta\to 0$ and
$$
\lim_{\e\to 0} \frac{\e}{\delta_\e}=0.
$$
then for all $u\in L^2(\Omega)$ there exists a sequence $u_\e$ converging strongly to $u$ in $L^2(\Omega)$ and such that 
$$
\lim_{\e\to 0}   F_{\delta,\eps}(u_\e)= 0.
$$
\end{Theorem}

 \begin{proof}
It is sufficient to prove that the claim holds for a strongly dense subclass in $L^2(\Omega)$.
We then take a Lipschitz continuous function $u$, and consider the function $u_\e$ equal to the constant
\begin{equation}\label{ue-ue}
u_\e(x)=u^\delta_k=\frac1{\delta^d|K|} \int_{\delta(k+K)} u(y)dy 
\end{equation}
in $\delta(k+K)$, and to $u$ on $\Omega\setminus \delta E$.  We have $u_\e\to u$ in $L^2(\Omega)$, and we can estimate
$$
F_{\delta,\eps}(u_\e)\le C \frac1{\e^{d+2}}\int_{(\Omega\times\Omega)\cap\{|x-y|>D_0\delta\}}\varphi\Big(\frac{|x-y|}\e\Big)|x-y|^2\,dx\,dy,
$$
 where $D_0=\min\{{\rm dist}(K, k+K): k\neq 0\}$. The change of variables $y=x+\e\xi$ gives
 $$
F_{\delta,\eps}(u_\e)\le C' \int_{\{|\xi|>D_0\frac\delta\e\}}\varphi(\xi)|\xi|^2d\xi,
$$
with $C'$ depending only on $\Omega$ and the Lipschitz constant of $u$.
% We consider a piecewise-constant function with a planar interface. It suffices to treat the case $u= \chi_{x_1>0}$, the characteristic function of the half-space of $x_1$ positive.
%We define 
%$$
%H_\delta=\{x\in\mathbb R^d: \hbox{$x_1>0$ or $x\in \delta (k+K)$ such that $(k+K)\cap\{x_1>0\}\neq\emptyset$}\}
%$$
%$$
%u_\e(x)= \begin{cases}
%1 & \hbox{ if $x\in H_\delta$}
%\cr
%0 & \hbox{ otherwise.}
%\end{cases}
%$$
%If $x,y\in \delta E$ and $u_\e(x)\neq u_\e(y)$ then $|x-y|\ge\delta D_0$,
%where $D_0=\min\{{\rm dist}(K, k+K): k\neq 0\}$. 
%We then have 
%$$
%F_{\delta,\eps}(u_\e)\le \frac1{\e^{d+2}}\int_{\{x\in\Omega:x_1>0\}}\int_{\{y\in\Omega:y_1<-\delta D_0\}} \varphi\big(\frac{|x-y|}\e\big)\,dx\,dy
%$$
%$$
%\le \frac1{\e^{d+2}}\int_{\{x\in\Omega:x_1>0\}}\int_{\{y\in\Omega:y_1<-\delta D_0\}} \varphi\big(\frac{|x_1-y_1|}\e\big)\,dx\,dy
%$$
This latter integral tends to $0$ by \eqref{crescfi}, since $\frac\delta\e\to +\infty$.
\end{proof}

 In the case $\delta\sim\e$ we havea lack of coerciveness in the case of kernels with compact support. Up to scaling, we can suppose that the support of $\varphi_0$ is $[0,1]$.
We then have the following result.

\begin{Proposition}\label{nonc-2}Let the support of $\varphi_0$ be $[0,1]$, and let $D$ be defined by \eqref{Di}.

{\rm(a)} If  $\delta =\delta_\e$ is such that $\e<\delta D_0$ then for all $u\in L^2(\Omega)$ there exists a sequence with $F_{\delta,\eps}(u_\e)=0$ and converging to $u$ strongly in $L^2(\Omega)$;

{\rm(b)}
 If $\delta =\delta_\e$ is such that $\e<\delta D$ then there exists a sequence 
 $u_\e$  with $F_{\delta,\eps}(u_\e)=0$ and such that $u_\e\chi_{\delta E}$ does not converge weakly in $L^2(\Omega)$. 
\end{Proposition}

\begin{proof}Case (a) is dealt with exactly as in the proof of Theorem \ref{no-comp-1}.
In case (b) we have that the set $\delta E+B_{\e/2}$ is composed of infinitely many disconnected components. We may suppose for simplicity that each connected component is not bounded since otherwise we are in case (a), and let $\overline k\in\mathbb Z^d$ be such that $\delta (\overline k+K)$ does not belong to the same connected component as $\delta K$. We may set $u_\e(x)= m$ if $x$ belongs to the same connected component as $\delta (2m \overline k+K)$, and for example $u_\e(x)=0$ elsewhere. Since $\e<\delta D$, if $\varphi(\frac{|x-y|}\e)>0$ and $x,y\in \delta E$ then they belong to the same connected component. This shows that $F_{\delta,\eps}(u_\e)=0$.
\end{proof}

\begin{Remark}\label{rem-deg}\rm
Claim (a) in Proposition \ref{nonc-2} shows that the $\Gamma$-limit in the strong $L^2$ topology is $0$. In the second case, if $\delta D_0<\e<\delta D$
then the sequence $F_{\delta,\eps}$ retains some form of coerciveness, which gives a non-trivial $\Gamma$-limit in the weak $L^2$ topology. For example, in the case of ellipsoidal sets as in Example \ref{Ex1}, the domain of the $\Gamma$-limit will be functions in $L^2(\Omega)$ whose distributional derivative in the $x_1$-direction is an $L^2$ function. Since this issue is not central in our discussion we omit the details of this case.
\end{Remark}

The following theorems will be proved in Section \ref{compactness}. They involve piecewise-affine interpolations obtained using \emph{Kuhn's decomposition} \cite{kuhn} of cubes of $\R^d$ into $d!$ simplexes which are uniquely determined by a permutation of the indices $\{1,...,d\}$. Since the actual form of the piecewise-affine interpolations is not relevant in our context we refer e.g.~to \cite{solci2023nonlocalinteraction}
for their use in nonlocal interaction problems.

\begin{Theorem}\label{comp1}
Let $\Omega$ be a connected set, let 
$$
\lim_{\e\to0} \frac\e\delta= \kappa,
$$
and suppose either that $\varphi_0$ has support $[0,1]$ and $\kappa>D$, or that the support 
of $\varphi$ be the whole $\mathbb R^d$.  Let $u_\e$ be such that $\sup_\e F_{\delta,\eps}(u_\e)<+\infty$.
Let the functions $\overline u_\e$ be defined 
as the piecewise-affine interpolations of the functions $\delta k\mapsto u^\delta_k$ in  \eqref{def convergence delta sim eps}  if $k\in \mathbb Z^d$ and $\delta (k+K)\subset \Omega$. Then, up to subsequences and addition of constants $\overline u_\e\to u$ in $L^2_{\rm loc}(\Omega)$ for some $u\in H^1(\Omega)$.
\end{Theorem}

The second result combines the interpolation argument with `coarse graining'; that is, we consider averages of functions not on the characteristic period $\delta$ of the geometry, but on the larger `mesoscopic' scale $\e$.
 
 \begin{Theorem}\label{comp2}
Let $\Omega$ be a connected set, let 
$$
\lim_{\e\to0} \frac\e\delta= +\infty.
$$
Let $u_\e$ be such that $\sup_\e F_{\delta,\eps}(u_\e)<+\infty$.
Let the functions $\overline u_\e$ be 
%defined as the piecewise-constant interpolations of the functions
defined as the piecewise-affine interpolations of the functions $\e k\mapsto u^\e_k$ defined on $\e\mathbb Z^d$ by
\begin{equation}
u^\e_k=\frac1{|\delta E\cap \e(k+[0,1]^d)|}\int_{\delta E\cap \e(k+[0,1]^d)} u_\e dx.
\end{equation}
Then, up to subsequences and addition of constants, $\overline u_\e\to u$ in $L^2_{\rm loc}(\Omega)$ for some $u\in H^1(\Omega)$.
\end{Theorem}

 \subsection{$\Gamma$-convergence}
The compactness results in the previous section allow to consider $\Gamma$-convergence with respect to the convergence of the interpolations as defined therein. We can then compute the $\Gamma$-limits in the hypotheses of Theorem \ref{comp1} and Theorem \ref{comp2}, respectively, the degenerate cases being dealt with in Theorem \ref{no-comp-1} when $\e<\!<\delta$, in which case the $\Gamma$-limit is $0$ with respect to any topology weaker than the strong $L^2(\Omega)$ one, and in Proposition \ref{nonc-2} for the case $\e< \delta D$, when $\varphi_0$ has support $[0,1]$.

\begin{Theorem}
Let  $\delta=\delta_\e$ satisfy
$$
\lim_{\e\to0} \frac\e\delta= \kappa,
$$
and suppose either that $\varphi_0$ has support $[0,1]$ and $\kappa>D$, or that the support 
of $\varphi$ be the whole $\mathbb R^d$. Then the $\Gamma$-limit of $F_{\delta,\eps}$ as $\e\to 0$ with respect to the convergence of the piecewise-affine interpolations as in Theorem {\rm\ref{comp1}} is the quadratic form 
\begin{equation}
F^\kappa(u)=\int_\Omega \langle A^\kappa_{\rm hom} \nabla u, \nabla u\rangle\,dx,
\end{equation}
where the symmetric matrix $A^\kappa_{\rm hom}$ satisfies
\begin{equation}
\langle A^\kappa_{\rm hom}\xi,\xi\rangle=
\min\Bigl\{\int_{[0,1]^d\cap E}\int_{E} \varphi(\tfrac{x-y}\kappa)(\langle\xi, x-y\rangle+ u(x)-u(y))^2\,dx\,dy:  u \hbox{ $1$-periodic}\Big\}.
\end{equation}
\end{Theorem}

\begin{proof}
Once the compactness in Theorem \ref{comp1} is proved, the proof of the claim follows very closely that of the homogenization theorem for perforated domains in \cite{braides2019homogenization}, where functionals of the same form as $F_{\delta,\eps}$ are dealt with, with $\e=\kappa\delta$ and $E$ a periodic (topologically) connected Lipschitz domain. We refer to that paper for details.
\end{proof}

The second convergence result highlights a separation of scales, in which we may formally first let $\delta$ tend to $0$ and note that the characteristic functions $\chi_{\delta E\times \delta E}$ weakly$^*$ converge to the constant $|K|^2$ in $L^\infty(\Omega\times\Omega)$.

\begin{Theorem}\label{sepscales}
Let  $\delta=\delta_\e$ satisfy
$$
\lim_{\e\to0} \frac\e\delta= +\infty.
$$
 Then the $\Gamma$-limit of $F_{\delta,\eps}$ as $\e\to 0$ with respect to the convergence of the piecewise-affine interpolations as in Theorem {\rm\ref{comp2}} is 
\begin{equation}
F^\infty(u)=|K|^2C_\varphi\int_\Omega|\nabla u|^2\,dx,
\end{equation}
where $C_\varphi$ is given by \eqref{ci-fi}
\end{Theorem}

 \begin{Remark}\label{l2conv}\rm
 We note that both the convergences described in the theorems above can be restated as a  local strong $L^2$-convergence; namely, that we have
 \begin{equation}
 \lim_{\e\to 0} \int_{\Omega'\cap \delta E} |u_\e-u|^2dx=0.
 \end{equation}
 This convergence has been extensively used by Zhikov \cite{zhikov1996connectedness}.
 
To check this claim, we note that by the Poincar\'e inequality for double integrals, for which we do not need the connectedness of the domain (see e.g.~the proof of \cite[Theorem 6.33]{Leoni}), taking into account the definition in Theorem \ref{comp2}. we have 
 \begin{eqnarray*}
 \int_{\delta E\cap \e(k+[0,1]^d)}|u_\e-u^\e_k|^2dx\le \frac1{\e^d}\int_{(\delta E\cap \e(k+[0,1]^d))^2}|u_\e(x)-u_\e(y)|^2dx\,dy,
 \end{eqnarray*}
 so that, summing on $k$
  \begin{eqnarray*}
 \int_{\Omega'\cap \delta E}|u_\e-\widehat u_\e|^2dx\le \e^2 C F_{\delta,\e}(u_\e),
 \end{eqnarray*}
 where $\widehat u_\e$ is the piecewise-constant interpolation of $\{u^\e_k\}$. The claim then follows noting that  $\widehat u_\e- \overline u_\e$ tends to $0$ locally in $L^2(\Omega)$.  For  the definition in Theorem \ref{comp1} the argument is the same. 
 
 As a consequence of this remark, we can suppose that the sequence $u_\e$ tends to $u$ locally in $L^2(\Omega)$, upon substituting $u_\e$ with the function defined by $u_\e$ on $\Omega\cap \delta E$, and $u$ otherwise.
 \end{Remark}

%We want to investigate the $\Gamma$-convergence of the functional (\ref{funzionale generale}) for different relative scales of $\eps$ and $\delta$:
%\begin{itemize}
%    \item for $\eps \ll \delta$ we expect convergence to 0, as $\eps$ is too small and so the balls do not interact with each other (thus one can take $u$ constant on each ball);
%    \item for $\delta \ll \eps$, we expect an interesting interaction, with the final form of the functional averaged in the "Riemann-Lebesgue lemma" sense, for example        $$        F_{\delta,\eps}(u) \longrightarrow | E \, \cap \Omega | \,C_{\phi} \cdot \int_{\Omega} |\nabla u|^2dx $$    where $C_{\phi}$ is a constant that appears if $\phi$ is radially symmetric.   \end{itemize}

\section{Compactness}\label{compactness}
In this section we prove Theorems \ref{comp1} and \ref{comp2}.
We make the choice of the kernel $\varphi=\chi_{B_r}$, and we may suppose that $r=1$ up to a scaling argument.
%$\phi_{\eps}$ such that
%    $$    \frac{\phi_{\eps}(x-y)}{|x-y|^2}= \frac{1}{\eps^{d+2}}, \qquad \textit{if} \ |x-y|<\eps,    $$
%and $0$ otherwise, thus $\phi_{\eps}$ is the $\eps$ rescaled version of $\phi(\xi)=|\xi|^2 \chi_{B_r}(\xi)$, with $r=1$ for simplicity. 
With this choice functional (\ref{funzionale generale}) becomes
    \begin{equation}\label{funzionale generale 2}
    F_{\delta,\eps}(u) = \frac{1}{\eps^{d+2}} \int_{{\substack{ (\Omega \cap \delta E)^2\\ \{|x-y|\}<\eps   }}} |u(x)-u(y)|^2dxdy.
    \end{equation}
Note that this $\phi$ is a lower bound for any other positive kernel satisfying the condition $\phi(\xi) \geq  c>0  $ for $\xi \in B_{r_0}$, 
so that it is sufficient to state the compactness result for families of functions $u_{\eps}$ such that (\ref{funzionale generale 2}) is bounded.

\bigskip\noindent
\textbf{Case 1:} $\delta \sim \e$ (Theorem \ref{comp1}). 
With fixed $D<\kappa ' < \kappa$, we can assume $\frac{\e}{\delta}>\kappa'$. By definition of $D$, there exist $r>0$ and a finite number of vectors $k_1,...,k_N \in \Z^d$ generating $\Z^d$ on $\Z$ such that dist$(K,K+k_i)<r<\kappa'$ for all $i\in\{1,...,N\}$.

 For each $i\in\{1,...,N\}$, let $z_i \in \partial K$ and $w_i \in \partial K + k_i$ be such that $|z_i-w_i|={\rm dist}\,(K,K+k_i)$. Moreover, set
$$
A_i :=\{x \in K : |x-z_i|< \tfrac{1}{2}(\kappa'-r)\} \quad \text{and} \quad B_i :=\{x \in K+k_i : |x-w_i|< \tfrac{1}{2}(\kappa'-r)\}.
$$
By the assumptions on $K$, we have $|A_i|, |B_i|>0$. Moreover, if $x \in \delta A_i$ and $y \in \delta B_i$
$$
\frac{1}{\delta}|x-y|\le |\tfrac{1}{\delta}x-z_i| + |z_i - w_i|+|w_i - \tfrac{1}{\delta}y | < \kappa' - r +\text{dist}(K,K+k_i) < \frac{\eps}{\delta};
$$
that is, $\delta A_i \times \delta B_i \subseteq (\delta K \times \delta(K+k_i)) \cap \{|x-y|<\eps\}$.
With this observation in mind, let $u \in L^2(\Omega)$ and recall the notation 
$$u_k=\frac{1}{\delta^d |K|} \int_{\delta K+\delta k} u(x)dx. 
$$ We compute 
    \begin{eqnarray*}
   && |u_0-u_{k_i}|^2= \Big|\frac{1}{\delta^d |K|} \int_{\delta K} (u(x)-u(x+\delta k_i))dx \Big|^2  \\
    & =&\Big|\frac{1}{\delta^d |K| |\delta A_i| |\delta B_i|} \int_{\delta K\times \delta A_i\times \delta B_i } (u(x)-u(x+\delta k_i)+u(y)-u(y)+u(z)-u(z))dxdydz \Big|^2 \\ 
    &\leq&
     3 \bigg( \Big|\frac{1}{\delta ^ {2d} |K| \, |A_i|} \int_{\delta K \times \delta A_i} (u(x)-u(y))dxdy \Big|^2 + \Big|\frac{1}{\delta^ {2d}|A_i| \, |B_i|} \int_{\delta A_i \times \delta B_i } (u(y)-u(z))dydz \Big|^2  \\
     &&+ \Big|\frac{1}{\delta ^ {2d} |K| \, |B_i|} \int_{\delta (K+k_i)\times \delta B_i} (u(z)-u(w))dwdz \Big|^2  \bigg)  \\ &\leq&
     \frac{C}{\delta^{2d}} \Big(\int_{\delta K \times \delta A_i} |u(x)-u(y)|^2dxdy +\int_{\delta A_i \times \delta B_i } |u(y)-u(z)|^2dydz  \\ && +  \int_{\delta (K+k_i)\times \delta B_i} |u(z)-u(w)|^2dwdz \Big)\\
     &\leq& \frac{C}{\delta^{2d}} \Big(\int_{\delta K \times \delta K} |u(x)-u(y)|^2dxdy +\int_{\substack{\delta K \times \delta (K+k_i)\\ \{|x-y|<\eps\}}} |u(y)-u(z)|^2dydz  \\ 
     && +  \int_{\delta (K+k_i)\times \delta (K+k_i)} |u(z)-u(w)|^2dwdz \Big),
    \end{eqnarray*}
where in the first inequality we used the relation $(a+b+c)^2 \leq 3(a^2+b^2+c^2)$, and in the second one Jensen's inequality and the fact that $|\delta A_i|=\delta^d |A_i|$ and $|\delta B_i|=\delta^d |B_i|$. 
%Consider the largest balls of interactions $A_{\delta},A'_{\delta}$ between two nearest neighbours $B_{c\delta}, B_{c\delta}(\delta e_1)$ in one direction chosen, here depicted in the direction $e_1$: 
\begin{figure}[h!]
\centering
\begin{tikzpicture}[domain=0:2]
    \draw[solid,fill=gray, fill opacity = 0.2] (0,0) circle (0.5);
    \draw[solid,fill=gray, fill opacity = 0.2] (2,0) circle (0.5);
    \draw [blue,dashed,domain=-25:25] plot ({2.1*cos(\x)}, {2.1*sin(\x)});
    \draw[solid,fill=black, fill opacity = 0.2] (1.79,0) circle (0.26);
    \draw [blue,dashed,domain=155:205,shift={(2,0)}] plot ({2.1*cos(\x)}, {2.1*sin(\x});
    \draw[solid,fill=black, fill opacity = 0.2] (0.195,0) circle (0.26);
    \node at (1.2,1.2) {$|x-y|< \eps$};
    \draw (0.195,0) -- (0.6,-1);
    \draw (1.79,0) -- (1.5,-1);
    \node at (0.6,-1.2) {$\delta A_i$};
    \node at (1.5,-1.2) {$\delta B_i$};
\end{tikzpicture}
\caption{interacting portions of inclusions}
\label{Fig2}\end{figure}
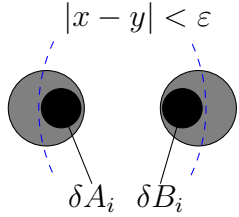
\begin{Example}\rm
    In the prototypical case $K=\overline B_c$, we can consider the largest balls fully contained in the region of interaction, as depicted in Fig.~\ref{Fig2}
\end{Example}
In order to take into account the first and the third integral of the sum above, we need to control long-range interactions with short-range interactions as specified in the following lemma.

\begin{Lemma}\label{lemma compactness}
There exists a constant $C$ depending on the ratio $\delta/\eps$  such that
    $$
    \int_{\delta K \times \delta K} |u(x)-u(y)|^2dxdy \leq C \int_{\substack{\delta K \times \delta K \\ \{|x-y|<\eps \} }} |u(x)-u(y)|^2dxdy 
    $$
\end{Lemma}
\begin{proof}
The proof follows the one of Lemma 6.1 in \cite{alicandro2023variational}. 
Since $K$ is a connected Lipschitz bounded set, there exists $\eps_1 \in \left(0,\frac{\eps}{3}\right)$ such that for every two points $\eta', \eta'' \in \delta K$ there exists a discrete path from $\eta'$ to $\eta''$; that is, a set of points $\eta'=\eta_0, \eta_1,\ldots,\eta_N,\eta_{N+1}=\eta''$ for which:
\begin{itemize}
    \item $|\eta_{j+1}-\eta_{j}| \leq \eps_1$ for $j\in\{0,1,\ldots,N\}$;
    \item for all $j=1,...,N$, the ball $B_{\eps_1}(\eta_{j})$ is contained in $\delta K$ (note that the two extremes are excluded);
    \item $N \leq C \left\lfloor \frac{ \delta}{\eps}\right \rfloor $ for all $\eta',\eta'' \in \delta K$. The constant $C$ depends only on $\text{diam}(K)$ and its Lipschitz constant;
\end{itemize}
%This follows for example by first joining $\eta'$ to the centre of $B_{c\delta}$ and then the centre to $\eta''$, for every two points chosen. 
Writing
    $$
    u(x_0)-u(x_{N+1})=u(x_0)-u(x_1)+u(x_1)-\cdots-u(x_N)+u(x_N)-u(x_{N+1}),
    $$
where $x_0 \in \delta K \cap B_{\eps_1}(\eta')$, $x_{N+1} \in \delta K \cap B_{\eps_1}(\eta'')$, and $x_j$ is a point in $B_{\eps_1}(\eta_j)$ for $j=1,..,N$. By integrating in all variables we get:
    \begin{align*}
    &\int_{(\delta K \cap B_{\eps_1}(\eta_0) ) \times (\delta K \cap B_{\eps_1}(\eta_{N+1}))} (u(x_0)-u(x_{N+1}))^2dx_0 dx_{N+1} \\
    &=(\eps_1)^{-Nd} \int_{B_{\eps_1}(\eta_1)}\cdots \int_{B_{\eps_1}(\eta_{N})} \int_{(\delta K \cap B_{\eps_1}(\eta_0) ) \times (\delta K \cap B_{\eps_1}(\eta_{N+1}))} \big(u(x_0)-u(x_1)+u(x_1)-\cdots \\
    &\hskip6cm \cdots -u(x_N)+u(x_N)-u(x_{N+1})\big)^2 dx_0 dx_{N+1} dx_1...dx_N \\ &\leq (N+1)(\eps_1)^{-Nd}\int_{\delta K \cap B_{\eps_1}(\eta_0)} \cdots \int_{\delta K \cap B_{\eps_1}(\eta_{N+1}) } \sum_{j=1}^{N+1}  ( u(x_{j})-u(x_{j-1}) )^2  dx_0\cdots dx_{N+1} \\
&\leq (N+1) \sum_{j=1}^{N+1} \int_{(\delta K \cap B_{\eps_1}(\eta_j) ) \times (\delta K \cap B_{\eps_1}(\eta_{j-1}))} (u(x_j)-u(x_{j-1}))^2dx_j dx_{j-1} \\
&\leq \left(C \left\lfloor \frac{ \delta}{\eps}\right \rfloor+1\right)^2 \int_{\substack{\delta K \times \delta K \\ \{|x-y|<\eps \} }} |u(x)-u(y)|^2dxdy,
\end{align*}
where in the last line by set inclusion we used that for any $(x_j, x_{j-1}) \in B_{\eps_1}(\eta_j) \times B_{\eps_1}(\eta_{j-1})$ holds
$$
|x_j-x_{j-1}|\leq |x_j-\eta_j|+|\eta_j-\eta_{j-1}|+|\eta_{j-1}-x_{j-1}|\leq 3\eps_1<\eps
$$
and that $N$ is at most $C \left\lfloor \frac{ \delta}{\eps}\right \rfloor$. We can obtain the final estimate by covering $\delta K$ with a finite number of balls of radius $\eps_1$ (this number depends on the ratio $\delta/\eps$ only) and then summing up the previous inequality applied to all possible pairs of these balls.
\end{proof}

By applying Lemma \ref{lemma compactness}, we deduce that there exists a constant $C$ such that, for $i=1,\ldots,d$,
 \begin{multline*}
    |u_0-u_{k_i}|^2\leq
    \frac{C}{\delta^{2d}} \bigg(\int_{\substack{\delta K \times \delta K \\ \{|x-y|< \eps \}}} |u(x)-u(y)|^2dxdy +\int_{{\substack{ \delta K \times \delta (K + k_i) \\ \{|x-y|\}<\eps   }}} |u(y)-u(z)|^2dydz 
    \\
    +\int_{\substack{ \delta (K + k_i) \times \delta (K + k_i) \\ \{|x-y|<\eps \}}} |u(z)-u(w)|^2dwdz \bigg).
    \end{multline*}
Since we need an estimate of 
$\delta^d \big(\frac{u_0-u_{k_i}}{\delta} \big)^2$,
we multiply by $\delta^{d-2}$ and use that $\frac{C}{\delta^{2d}}\delta^{d-2}= \frac{C}{\eps^{d+2}}$ (by $ \delta \sim \eps $) to find

    $$
    \delta^d \left(\frac{u_0-u_{k_i}}{\delta} \right)^2 \leq \frac{C}{\eps^{d+2}} \int_{{\substack{(\delta K \, \cup \,  \delta(K+k_i))^2 \\ \{|x-y|\}<\eps } }} |u(x)-u(y)|^2 dxdy
    $$
By periodicity it also holds
    $$
    \delta^d \left(\frac{u_k-u_{k+k_i}}{\delta} \right)^2 \leq \frac{C}{\eps^{d+2}} \int_{{\substack{(\delta (K+k) \, \cup \, \delta(K+(k+k_i)))^2 \\ \{|x-y|\}<\eps } }} |u(x)-u(y)|^2 dxdy.
    $$
Finally, since each vector of the canonical basis can be expressed as a linear combination of the $k_i$'s with integer coefficients, via triangle inequality we bound
$$
    \delta^d \left(\frac{u_k-u_{k+e_i}}{\delta} \right)^2 \leq \frac{C}{\eps^{d+2}} \int_{{\substack{(\delta (K+k) \, \cup \, \delta(K+(k+k_i)))^2 \\ \{|x-y|\}<\eps } }} |u(x)-u(y)|^2 dxdy,
    $$
and thus, summing over nearest neighbours $k,k' \in \Z^d$ such that $k\delta + \delta K \subseteq \Omega $, and the same for $k'$, we get
    $$
   \sum_{\substack{ k,k' \\ |k-k'|=1}} \delta^{d} \left(\frac{u_k-u_{k'}}{\delta}\right)^2 \leq \frac{C}{\eps^{d+2}} \int_{{\substack{ (\Omega \cap \delta E)^2\\ \{|x-y|\}<\eps   }}} |u(x)-u(y)|^2dxdy.
    $$
Therefore, from the equiboundedness of the functional along the sequence $u_\eps$ we get a uniform bound on the piecewise-affine interpolation $\overline u_\e$, as defined in Theorem \ref{comp1}, and thus the existence of a converging subsequence in the sense of (\ref{def convergence delta sim eps}). 

\bigskip
\textbf{Case 2}: $\eps \gg \delta$  (Theorem \ref{comp2}). 
%If we discretize the space $\Omega$ in cubes of side length $\frac{1}{2\sqrt{d}}\eps$, then we are sure that, when $|x-y|<\eps$, every point of each cube interacts with each of its nearest neighbours, so do the balls $B_{c\delta}$ inside them (which are of order $\left \lfloor \frac{\eps}{\delta} \right \rfloor^d$). 
%Let   $$   \mathcal{J}_{\eps} = \left\{ j \in \Z^d: Q_{\eps}^{j} \subset \Omega \right\}   $$
%where $Q^{j}_{\eps}= Q_{\frac{\eps}{2\sqrt{d}}}\left(\frac{\eps}{\sqrt{d}}j\right)$.   
Set $C=C(d) := \frac{1}{1+\sqrt{d}}$, $Q=[0,1]^d$ and
   $$
   \mathcal{Z}_{\eps} = \big\{ k \in \Z^d:  \overline{Q_\eps ^k}  := C\eps (k + \overline Q)  \subseteq \Omega \big\}.
   $$
Note that if $(x,y) \in Q_\eps^k \times Q_\eps^{k'}$ with $|k-k'| \le 1$; that is $x=C\eps (k + w)$ and $y=C\eps(k'+ z)$ with $w,z \in Q$, it automatically holds
$$
|x-y| \le C\eps |k-k'| + C\eps |w-z| \le C\eps \left(1+\text{diam}(Q)\right) = \eps.
$$
Moreover, for each $k \in Z_\eps$, consider
$$
Z_k  := \{j \in \mathbb{Z}^d : \delta(j+K) \subseteq Q_\eps^k\}= \Big\{j \in \mathbb{Z}^d : j+K \subseteq \frac{\eps}{\delta} (k+ Q)\Big\}.
$$
Note that 
$\delta E \cap Q_\eps^k \supseteq \bigcup_{j \in Z_k} \delta(j+K) \cap Q_\eps^k= \bigcup_{j \in Z_k} \delta(j+K)$.
 Moreover, assume $K \subseteq B_L$ and fix $j \in \frac{\eps}{\delta} k + B(0, \frac{\eps}{2\delta}-L)$ and $ x = j +  w \in j + K$; then for all $i\in\{1,\ldots,d\}$ we have
$$
|x_i - \tfrac{\eps}{\delta} k_i | \le |x- \tfrac{\eps}{\delta} k| \le |j -  \tfrac{\eps}{\delta} k| + L \le \frac{\eps}{2\delta};
$$
that is, $j+K \subseteq \frac{\eps}{\delta} (k+ Q)$. We deduce that $\frac{\eps}{\delta} k+ \{- \lfloor{ \frac{\eps}{2\delta}-L}\rfloor,...,0,..., \lfloor{ \frac{\eps}{2\delta}-L}\rfloor\}^d \subseteq Z_k$ and consequently $\#Z_k \ge (2(\frac{\eps}{\delta}-L-1) + 1) ^d \ge \left(\frac{\eps}{\delta}\right)^d$, provided $\frac{\eps}{2\delta} \ge L+1.$
In particular 
$$
|\delta E \cap Q_\eps^k| \ge \sum_{j \in Z_k} |\delta (j+K)| =\#Z_k |K| \delta^d \ge  |K| \eps^d.
$$

We now write
\begin{align*}
     F_{\eps}(u_{\eps})&=\frac{1}{\eps^{d+2}} \int_{{\substack{ (\Omega \cap \delta E)^2\\ \{|x-y|\}<\eps   }}} (u_{\eps}(x)-u_{\eps}(y))^2dxdy \\&\geq \frac{1}{\eps^{d+2}} \sum_{\substack{|k-k'|=1 \\ k,k' \in \mathcal{Z}_{\eps}}} \int_{(Q^k_{\eps} \times Q^{k'}_{\eps}) \cap (\delta E)^2} (u_{\eps}(x) - u_{\eps}(y))^2 dxdy \\
     & = \frac{1}{\eps^{d+2}} \sum_{\substack{|k-k'|=1 \\ k,k' \in \mathcal{Z}_{\eps}}} |\delta E \cap Q_\eps^k||\delta E \cap Q_\eps^{k'}| \fint_{(Q^k_{\eps} \times Q^{k'}_{\eps}) \cap (\delta E)^2} (u_{\eps}(x) - u_{\eps}(y))^2 dxdy \\
     &\ge      \frac{|K|}{\eps^{d+2}} \sum_{\substack{|k-k'|=1 \\ k,k' \in \mathcal{Z}_{\eps}}} \eps^{2d} \left| \fint_{(Q_\eps^k \cap \delta E)\times (Q_\eps^{k'} \cap \delta E)} \left(u_\eps(x) - u_\eps(y)\right) dx dy  \right|^2 \\
     &= \frac{|K|}{\eps^{d+2}} \sum_{\substack{|k-k'|=1 \\ k,k' \in \mathcal{Z}_{\eps}}} \eps^{2d} \left|  u_\eps^k   - u_\eps^{k'}   \right|^2 \\     
     &=|K| \sum_{\substack{|k-k'|=1 \\ k,k' \in \mathcal{Z}_{\eps}}} \eps^d \frac{|u^k_{\eps}-u^{k'}_{\eps}|^2}{\eps^2},
\end{align*}
with the averages $u^{k}_{\eps}$ defined by
  $$
  u^{k}_{\eps}= \frac{1}{|Q_{\eps}^k \cap \delta E | }\int_{Q_{\eps}^k \cap \delta E} u_{\eps}(x)dx.
  $$
As in the previous case, we get a uniform bound on the piecewise-affine interpolation $\overline{u}_{\eps}$, as defined in Theorem \ref{comp2}, so we can conclude that there exists a subsequence converging in the sense of (\ref{def convergence delta ll eps}).

\section{Asymptotic analysis}
This section is devoted to the proof of Theorem \ref{sepscales}, subdividing it in a lower and an upper bound.

\subsection{Liminf inequality}\label{subsection liminf inequality}
We first prove a lower bound. To that end, let $u_{\eps} \to u$ in the sense of (\ref{def convergence delta ll eps}) with $F_{\eps}(u_{\eps}) \leq M < +\infty$. 
Using the conclusion in Remark \ref{l2conv}  we can suppose that the sequence $u_\e$ tends to $u$ locally in $L^2(\Omega)$, and actually that $u_\e=u$ outside $\delta E$.

In order to estimate the energy we use Fonseca-M\"uller's blow-up technique \cite{FM}, as adapted to homogenization problems in \cite{BMS}.
To that end, we restrict $x \in A \subset \Omega$, where $A$ is an open subset, defining
    \begin{equation}\label{def of mu_eps}
     F_{\eps}(u_{\eps},A)= \frac{1}{\eps^{d+2}} \int_{  \substack{(\boldsymbol{A} \times \Omega) \cap (\delta E)^2 \\ \{|x-y|<\eps\}}}|u_{\eps}(x)-u_{\eps}(y)|^2dxdy=\mu_{\eps}(A)
    \end{equation}
Since the measures $\mu_{\eps}$ are equibounded in mass, we can take a converging subsequence $\mu_{\eps} \overset{\ast}{\rightharpoonup} \mu$,
so that by the lower semicontinuity of weak$^{\ast}$ convergence on open sets it holds
    $$
    \liminf_{\eps \to 0} F_{\eps}(u_{\eps}) \geq \mu(\Omega) \geq \int_{\Omega} \frac{d \mu}{dx} dx.
    $$    
From now on, we denote $Q := (-\frac{1}{2}, \frac{1}{2})^d$ and, for all $r>0$ and $x_0 \in \R^d$, $Q_{\rho}(x_0) := x_0 + \rho Q$.
We fix a Lebesgue point $x_0 \in \Omega$ for $u$ and $\nabla u$, such that there exists 
    $$
    \frac{d\mu}{dx}(x_0)= \lim_{\rho \to 0} \frac{\mu(Q_{\rho}(x_0))}{\rho^d}.
    $$
Since $\mu (\partial Q_{\rho}(x_0) )=0$ for almost all $\rho >0$, we can construct ${\rho}_\e \to 0$, with ${\rho}_\e \gg \e$, such that
  \begin{equation}\label{measure quotient}
    \frac{d\mu}{dx}(x_0)= \lim_{\e \to 0} \frac{\mu_\e(Q_{\rho_\e}(x_0))
  }{\rho_\e^d}.
  \end{equation}
From now on we will tacitly suppose
that $\rho=\rho_\e$. Now we prove a classical lemma that allows to match boundary data. 
\begin{Lemma}\label{match boundary data lemma}
Let $\{v_\e\} \subseteq L^2(Q)$ such that $v_\e \to v$ for some $v\in H^1(Q)$. For all $\eta >0$, there exists another sequence $\{v^\eta_\e\} \subseteq L^2(Q)$, such that 
\begin{itemize}
    \item $v^\eta_\eps \to v$ in $L^2(Q)$
    \item $v^\eta_\eps =v $ in the set $\{ z \in Q: \emph{dist}(z,\partial Q) < \eta \}$
    \item $v^\eta_\eps =v_\eps $ in the set $\{ z \in Q: \emph{dist}(z,\partial Q) > 2\eta \}$
    \item if we define $$
    F'_{\eps}(v_{\eps})=\frac{1}{\eps^{d+2}} \int_{  \substack{(Q\times Q) \cap (\delta E)^2 \\ \{|z-w|<\eps\}}}\left(v_\eps(z)-v_\eps(w)\right)^2dz dw,
    $$
    then it holds
         $$
          \limsup_{\eps \to 0} (F'_{\eps}(v^\eta_\eps)-F'_{\eps}(v_\eps ) )\leq o(1),
           $$
           as $\eta \to 0$.
\end{itemize}
\end{Lemma}
\begin{proof}
  Fix $N \in \mathbb{N}$ and for $k \in \{1,...,N\}$ define $Q_k^N := \{ z \in Q \colon \text{dist}(z, \partial Q) > \frac{\eta}{N}(N+k)\}$ and $S_k  := Q_{k-1}^N \setminus \overline{Q_{k}^N}$. Let $\pphi_k$ be a cut-off function such that $\pphi_k=1$ in $Q_k^N$, $\pphi_k=0$ in $\R^d \setminus \overline{Q_{k-1}^N}$ and $|\nabla \pphi_k| \le \frac{N}{\eta}$. Define
  $$
  v_\e ^k  := \pphi_k v_\eps + (1-\pphi_k) v.
  $$
  Note that
$v_\e^k (z) - v_\e^k (w)= \pphi_k(z)(v_\e (z) - v_\e (w)) + (1-\pphi_k(z)) (v(z)-v(w) )(\pphi_k(z) - \pphi_k(w)) (v_\e(w)-v(w))$ for all $z, w \in Q$.

%\begin{equation}\label{functiona in the lemma}
% \frac{1}{\eps^{d+2}} \int_{  \substack{(Q\times Q) \cap (\delta E)^2 \\ \{|z-w|<\eps\}}}\left(v_\eps^k(z)-v_\eps^k(w)\right)^2dz dw
%\end{equation}
We decompose $Q \times Q$ as follows:
\begin{align*}
&A_1 := Q_k^N \times Q_k^N,  \quad \quad \pphi_k(z)=1, \ \pphi_k(w)=1 \ \forall (z,w) \in A_1;\\
&A_2 := (Q \setminus \overline{Q_{k-1}^N})\times (Q \setminus \overline{Q_{k-1}^N}), \quad \quad \pphi_k(z)=0, \ \pphi_k(w)=0 \ \forall (z,w) \in A_2\\
&A_3 := S_k \times Q, \quad \quad |\nabla \pphi_k(z)|\leq 1 \\
&A_3' := \left(Q \times S_k \right) \setminus A_3 \\
&A_4  := Q_k^N \times (Q \setminus \overline{Q_{k-1}^N}), \quad \quad |z-w| \geq \frac{\eta}{N},  \forall (z,w) \in A_2 \\
&A_4'  := \big(Q \setminus \overline{Q_{k-1}^N} \big)\times Q_k^N. 
\end{align*}
The integral in $A_1$ is estimated by:
\begin{align*}
 &\frac{1}{\eps^{d+2}} \int_{  \substack{A_1 \cap (\delta E)^2 \\ \{|z-w|<\eps\}}}\left(v_\eps^k(z)-v_\eps^k(w)\right)^2dz dw\\
 &=\frac{1}{\eps^{d+2}} \int_{  \substack{A_1 \cap (\delta E)^2 \\ \{|z-w|<\eps\}}}\left(v_\eps(z)-v_\eps(w)\right)^2dz dw \leq \frac{1}{\eps^{d+2}} \int_{  \substack{(Q \times Q) \cap (\delta E)^2 \\ \{|z-w|<\eps\}}}\left(v_\eps(z)-v_\eps(w)\right)^2dz dw= F'_{\eps}(v_{\eps})
\end{align*}
Since $Q$ has a Lipschitz boundary and $v \in H^1(Q)$, we can extend it to the whole $\R^d$. Denote the set 
$$
S_\eta=\{ z \in Q: \text{dist}(z,\partial Q) < 2 \eta\}.
$$
We have
\begin{multline*}
 \frac{1}{\eps^{d+2}} \int_{  \substack{A_2 \cap (\delta E)^2 \\ \{|z-w|<\eps\}}}\left(v_\eps^k(z)-v_\eps^k(w)\right)^2dz dw =\frac{1}{\eps^{d+2}} \int_{  \substack{A_2 \cap (\delta E)^2 \\ \{|z-w|<\eps\}}}\left(v(z)-v(w)\right)^2dz dw \\
 \leq \frac{1}{\eps^{d+2}} \int_{  \substack{ (S_\eta \times S\eta) \cap (\delta E)^2 \\ \{|z-w|<\eps\}}}\left(v(z)-v(w)\right)^2dz dw \leq
 \frac{1}{\eps^{d+2}} \int_{  \substack{ (S_\eta \times S_\eta) \\ \{|z-w|<\eps\}}}\left(v(z)-v(w)\right)^2dz dw  \\ 
 = \int_{|\xi|<1} |\xi|^2 \int_{S_{\eta,\xi}}\left( \frac{v(z+\eps \xi)-v(z)}{\eps |\xi|} \right)^2dz d\xi , 
\end{multline*}
where $S_{\eta,\xi}= \{z \in S_\eta: z+\eps \xi \in S_{\eta} \} $. Note that, for all $|\xi| <1 $ and $\eps$ small enough, the set $S_{\eta,\xi}$ is well included in $S_\eta+B_\eta(0)$. Thus we can use a well-known characterization of Sobolev space $H^1$ to estimate
$$
\int_{|\xi|<1} |\xi|^2 \int_{S_{\eta,\xi}}\left( \frac{v(z+\eps \xi)-v(z)}{\eps |\xi|} \right)^2dz d\xi \leq \int_{|\xi|<1} |\xi|^2 \|\nabla v \|_{L^2(S_\eta+B_\eta(0))} d\xi = \|\nabla v \|_{L^2(S_\eta+B_\eta(0))},
$$
which tends to 0 as $\eta \to 0$ independently of $\eps$. As for $A_3$, we have
\begin{align}
\frac{1}{\eps^{d+2}} \int_{  \substack{A_3 \cap (\delta E)^2 \\ \{|z-w|<\eps\}}}&\left(v_\eps^k(z)-v_\eps^k(w)\right)^2dz dw \label{stima v^k_eps}  \\
 &=\frac{1}{\eps^{d+2}} \int_{  \substack{A_3 \cap (\delta E)^2 \\ \{|z-w|<\eps\}}}(\pphi_k(z)\left(v_\e (z) - v_\e (w)\right) + \left(1-\pphi_k(z)\right) \left(v(z)-v(w) \right)\nonumber \\ &\ \ 
 +(\pphi_k(z) - \pphi_k(w)) (v_\e(w)-v(w)))^2dz dw  \nonumber \\ 
 &\le \frac{3}{\eps^{d+2}} \int_{  \substack{A_3 \cap (\delta E)^2 \\ \{|z-w|<\eps\}}}\left(v_\e (z) - v_\e (w)\right)^2 dz dw \label{first term of three} \\
 &\ \ +\frac{3}{\eps^{d+2}} \int_{  \substack{A_3 \cap (\delta E)^2 \\ \{|z-w|<\eps\}}}\left(v (z) - v(w)\right)^2 dz dw \label{second term of three}\\
 &\ \ +\frac{3}{\eps^{d+2}} \int_{  \substack{A_3 \cap (\delta E)^2 \\ \{|z-w|<\eps\}}}\left(\pphi_k (z) - \pphi_k(w)\right)^2(v_\eps(w)-v(w))^2 dz dw \label{third term of three}
\end{align}
Recalling that  $A_3=S_k \times Q$, we estimate (\ref{third term of three}) with  
\begin{eqnarray*}
  &&\hskip-2cm\int_{  \substack{(S_k \times Q) \cap (\delta E)^2 \\ \{|z-w|<\eps\}}}\left(\pphi_k (z) - \pphi_k(w)\right)^2(v_\eps(w)-v(w))^2 dz dw  \\ 
  &\le& \frac{N^2}{\eta^2}\int_{ \substack{ S_k \times Q \\  \{|z-w|<\eps\}}}|z - w|^2(v_\eps(w)-v(w))^2 dz dw \\
  & =& \frac{N^2}{\eta^2} \int_{S_k}  \left(\int_{  Q \cap B(w, \eps) }|z - w|^2 dz \right) (v_\eps(w)-v(w))^2 dw  \\
   &\le& \frac{N^2}{\eta^2} \int_{S_k}  \eps^2 |B_\eps| (v_\eps(w)-v(w))^2 dw  \\
   &=&\eps^{d+2} \frac{N^2}{\eta^2} \int_{S_k}   (v_\eps(w)-v(w))^2 dw .
\end{eqnarray*}
Note that $v_\eps \to v$ in $L^2(Q)$ by assumption, so the last term goes to $0$ as $\eps \to 0$. Moreover, when we sum over $k=1,...,N$ the terms (\ref{first term of three}) and (\ref{second term of three}), we find
\begin{align*}
    &\sum_{k=1}^N \bigg( \frac{1}{\eps^{d+2}} \int_{  \substack{A_3 \cap (\delta E)^2 \\ \{|z-w|<\eps\}}}\left(v_\e (z) - v_\e (w)\right)^2 dz dw + \frac{1}{\eps^{d+2}} \int_{  \substack{A_3 \cap (\delta E)^2 \\ \{|z-w|<\eps\}}}\left(v (z) - v(w)\right)^2 dz dw \bigg) \\
    &\le \frac{1}{\eps^{d+2}} \int_{  \substack{Q^2 \cap (\delta E)^2 \\ \{|z-w|<\eps\}}}\left(v_\eps(z)-v_\eps(w)\right)^2dz dw + \frac{1}{\eps^{d+2}} \int_{  \substack{Q^2 \cap (\delta E)^2 \\ \{|z-w|<\eps\}}}\left(v(z)-v(w)\right)^2dz dw.
\end{align*}
Thus we may find $k \in \{1,\ldots,N\}$ such that
\begin{eqnarray*}
    &&  \frac{1}{\eps^{d+2}} \int_{  \substack{A_3 \cap (\delta E)^2 \\ \{|z-w|<\eps\}}}\left(v_\e (z) - v_\e (w)\right)^2 dz dw + \frac{1}{\eps^{d+2}} \int_{  \substack{A_3 \cap (\delta E)^2 \\ \{|z-w|<\eps\}}}\left(v (z) - v(w)\right)^2 dz dw  \\
    &\le& \frac{1}{N} \bigg(\frac{1}{\eps^{d+2}} \int_{  \substack{Q^2 \cap (\delta E)^2 \\ \{|z-w|<\eps\}}}\left(v_\eps(z)-v_\eps(w)\right)^2dz dw + \frac{1}{\eps^{d+2}} \int_{  \substack{Q^2 \cap (\delta E)^2 \\ \{|z-w|<\eps\}}}\left(v(z)-v(w)\right)^2dz dw\bigg).
\end{eqnarray*}
With such a choice of $k$, and setting $v^\eta_\eps= v^k_\eps$, (\ref{stima v^k_eps}) is estimated by
\begin{eqnarray*}&&
 \frac{1}{\eps^{d+2}} \int_{  \substack{A_3 \cap (\delta E)^2 \\ \{|z-w|<\eps\}}}\left(v_\eps^\eta(z)-v_\eps^\eta(w)\right)^2dz dw  \\ &\le &\frac{3}{N} \bigg(\frac{1}{\eps^{d+2}} \int_{  \substack{Q^2 \cap (\delta E)^2 \\ \{|z-w|<\eps\}}}\left(v_\eps(z)-v_\eps(w)\right)^2dz dw + \frac{1}{\eps^{d+2}} \int_{  \substack{Q^2 \cap (\delta E)^2 \\ \{|z-w|<\eps\}}}\left(v(z)-v(w)\right)^2dz dw\bigg)  \\ 
 &&+3\frac{N^2}{\eta^2} \int_{S_k}   (v_\eps(w)-v(w))^2 dw= \frac{3}{N}( F'_{\eps}(v_{\eps}) +  F'_{\eps}(v) )  +3\frac{N^2}{\eta^2} \int_{S_k}   (v_\eps(w)-v(w))^2 dw \\ &\leq& \frac{C}{N} + 3\frac{N^2}{\eta^2} \int_{S_k}   (v_\eps(w)-v(w))^2 dw.
 \end{eqnarray*}
The estimate on $A_3'$ is the same with the roles of $z,w$ reversed.

As for $A_4$, as already noticed for $(z,w) \in A_4$ we have
$|z-w| \geq \frac{\eta}{N}$,
so that, for $\eta,N$ fixed and $\eps$ small enough (indeed we let it go to $0$ before arguing on $N$), $\frac{\eta}{N}<|z-w|< \eps$ cannot happen simultaneously and thus we simply have  
 $$
 \frac{1}{\eps^{d+2}} \int_{  \substack{A_4 \cap (\delta E)^2 \\ \{|z-w|<\eps\}}}\left(v_\eps^k(z)-v_\eps^k(w)\right)^2dz dw =0 \quad \text{for $\eps$ small enough}.   
 $$
This works the same for $A'_4$ as well. Finally, gathering all the estimates together, we find

$$
\limsup_{\eps \to 0} (F'_\eps(v^\eta_\eps)-F'_\eps(v_\eps )) \leq  \frac{C}{N} +\|\nabla v \|_{L^2(S_\eta+B_\eta(0))}. 
$$
Sending $N \to \infty$, we prove the claim since $\|\nabla v \|_{L^2(S_\eta+B_\eta(0))}=\omega(\eta)=o(1)$ as $\eta \to 0$. 
\end{proof}
Recalling \eqref{measure quotient}, we consider a Lebesgue point $x_0$ and a sequence $\rho=\rho_\e$ for which the limit 
$$
\frac{d\mu}{dx}(x_0)= \lim_{\e \to 0} \frac{\mu_\e(Q_{\rho}(x_0))
  }{\rho^d}.
$$
exists, with $\mu_\e$ is defined as in $(\ref{def of mu_eps})$.

\begin{Proposition}\label{lemma applied to u_eps}
Let $\{u_\e\} \subseteq L^2(\Omega)$, $ u_\e \to u \in H^1(\Omega)$, and let $\eta >0$ be fixed. There exists $\tilde u_\e\to u$ such that 
\begin{itemize}
    \item $\tilde u_\e(x)=u_\e (x_0) + \nabla u (x_0)\cdot (x-x_0), \  \ \hbox{for all } x \in Q_\rho(x_0) \setminus Q_{\rho(1-2\eta)}(x_0)$;
    \item $\tilde u_\e(x)=u_\e (x) \ \ \hbox{for all }x \in Q_{\rho(1-4\eta)}(x_0);$
    \item denoting for simplicity 
    $$F'_\e(u)= \frac{1}{\rho^d\eps^{d+2}} \int_{  \substack{(Q_\rho(x_0) \cap \delta E)^2 \\ \{|x-y|<\eps\}}}(u(x)-u(y))^2dxdy,
    $$ 
it holds
    $$
    \limsup_{\e \to 0} (F'_\e(u_\e)-F'_\e(\tilde u_\e)) \leq o(1)
    $$
    as $\eta \to 0$.
\end{itemize}
\end{Proposition}
\begin{proof} The proof of this proposition is obtained by applying Lemma \ref{match boundary data lemma} to the difference quotient $v_\e$ defined by 
$$
v_\e(z) :=\frac{u_\e (x_0 + \rho z)-u_\e(x_0)}{\rho}, \quad z \in Q.
$$
We may choose $\rho=\rho_\e$ such that also $
v_\e \to \nabla u(x_0)\cdot z \ \text{in } L^2(Q)$.
We rewrite the quotient $\frac{\mu_\e(Q_{\rho}(x_0))
  }{\rho^d}$ in terms of the functions $v_\eps$:
\begin{eqnarray*}
\frac{\mu_\e(Q_{\rho}(x_0))
  }{\rho^d}&= &\frac{1}{\rho^d\eps^{d+2}} \int_{  \substack{(Q_{\rho}(x_0)\times \Omega) \cap (\delta E)^2 \\ \{|x-y|<\eps\}}}(u_{\eps}(x)-u_{\eps}(y))^2dxdy \\
  &\geq& \frac{1}{\rho^d\eps^{d+2}} \int_{  \substack{(Q_{\rho}(x_0)\times Q_{\rho}(x_0)) \cap (\delta E)^2 \\ \{|x-y|<\eps\}}}(u_{\eps}(x)-u_{\eps}(y))^2dxdy\\&=&
    \frac{\rho^d}{\eps^{d+2}} \int_{  \substack{(Q\times Q) \cap (\frac{\delta}{\rho} E)^2 \\ \{|z-w|<\frac{\eps}{\rho}\}}}(u_{\eps}(x_0 + \rho z)-u_{\eps}(x_0 + \rho w ))^2dz dw\\&=&
     \frac{\rho^{d+2}}{\eps^{d+2}} \int_{  \substack{(Q\times Q) \cap (\frac{\delta}{\rho} E)^2 \\ \{|z-w|<\frac{\eps}{\rho}\}}}\left(\frac{u_{\eps}(x_0 + \rho z)-u_{\eps}(x_0 + \rho w )}{\rho}\right)^2dz dw \\&=&
     \frac{\rho^{d+2}}{\eps^{d+2}} \int_{  \substack{(Q\times Q) \cap (\frac{\delta}{\rho} E)^2 \\ \{|z-w|<\frac{\eps}{\rho}\}}}\left(v_\eps(z)-v_\eps(w)\right)^2dz dw.
    \end{eqnarray*}
By Lemma \ref{match boundary data lemma}, there exists $v_\e^\eta \to v$ in $L^2(Q)$ which satisfies the properties stated therein, in particular $v_\e^\eta(z)= \nabla u(x_0)\cdot z$ in the set $ \{ z \in Q: \text{dist}(z,\partial Q) < \eta \}$.
By setting $$\tilde u_\e(x_0 + \rho z) := \rho v_\e^\eta (z) + u_\e(x_0) \ \hbox{ for all z } \in Q,$$
the function $\tilde u_\e$ satisfies the properties in the statement; in particular, $\tilde u_\e(x)=u_\e (x_0) + \nabla u (x_0)\cdot (x-x_0)$ in the set $ \{x \in Q_\rho(x_0): \text{dist}(z,\partial Q_\rho(x_0))< \rho \eta \} = Q_\rho(x_0) \setminus Q_{\rho(1-2\eta)}(x_0)$, and by the same change of variable of above to write the functional back in terms of $\tilde u_\e$ it holds 
\begin{multline*}
\limsup_{\e \to 0} \Big( \frac{1}{\rho^d\eps^{d+2}} \int_{  \substack{(Q_{\rho}(x_0) \cap\delta E)^2 \\ \{|x-y|<\eps\}}}(u_{\eps}(x)-u_{\eps}(y))^2dxdy  \\ -  \frac{1}{\rho^d\eps^{d+2}} \int_{  \substack{(Q_{\rho}(x_0) \cap\delta E)^2 \\ \{|x-y|<\eps\}}}(\tilde u_{\eps}(x)-\tilde u_{\eps}(y))^2dxdy \Big) \leq o(1),
\end{multline*}
as $\eta \to 0$.
\end{proof}
To ease the notation we still call $u_\e$ the sequence which satisfies $u_\e(x)=u_\e(x_0)+\nabla u (x_0) \cdot (x-x_0)$ for $x \in Q_{\rho}(x_0) \setminus Q_{\rho(1 - 2\eta)}(x_0) $ of Proposition \ref{lemma applied to u_eps}. 

Now we would like to have $\frac{\rho}{\delta} \in \mathbb{N} $, so we take a slightly smaller cube than $Q_\rho(x_0)$ that gives a negligible error at the limit.
Set $N_\e  := \left \lfloor{\frac{1}{2}(\frac{\rho}{\delta}-1) }\right \rfloor$, $l_\e  := (1+2N_\e)\delta $ and
\begin{align*}
 Q_\e := Q_{l_\e} (x_0)=\bigcup _{k \in \{-N_\e, ..., 0,...,N_\e\}^d} Q_{\delta} (x_0 + \delta k), \qquad |Q_\e|=l_\e^d,
\end{align*} 
In the hypothesis $\eps \gg \delta$ and $\eta \sim \frac{\eps}{\rho}$, we have  $Q_{\rho(1- \eta)} (x_0) \subseteq Q_\e \subseteq Q_{\rho}(x_0) $. In particular, we can extend $u_\e - u_\e(x_0) -\nabla u(x_0)\cdot (\cdot - x_0)$ periodically of period $l_\e$ to the whole $\R^d$.

Note that for all $x \in Q_\e \subseteq Q_{\rho}(x_0)$ and $y \in B(x;\e)$, choosing $\eta= 2 \frac{\e}{\rho}$,
\begin{align*}
|y_i - (x_0)_i| \le |x_i - (x_0)_i| +|x-y| \le \frac{\rho}{2}+ \e = \frac{\rho(1 +\eta)}{2}  \  \hbox{ for all }i \in \{1,\ldots,d\};
\end{align*}
that is,  $y \in Q_{\rho(1+ \eta)}(x_0)$. 
Thus we have 
$$
(Q_\e \times \R^d) \cap \{|x-y|<\eps \} \subset (Q_\e \times Q_{\rho(1+ \eta)}(x_0)) \cap \{|x-y|<\eps \}.
$$
We further split the right-hand side in
\begin{equation}\label{splitting of Q eps x Q eps}
((Q_\e \times Q_\e) \cap \{|x-y|<\eps \})  \cup  (  (Q_\e \times (Q_{\rho(1+ \eta)}(x_0) \setminus Q_\e)) \cap \{|x-y|<\eps \}   ).
\end{equation}
For $y \in Q_{\rho(1+ \eta)}(x_0) \setminus Q_\eps \subset Q_{\rho(1+ \eta)}(x_0) \setminus Q_{\rho(1- \eta)}(x_0) $, and $x \in Q_\e \cap B(y,\e)$ we have
$$
\frac{\rho(1-\eta)}{2} < |y_i - (x_0)_i| \le |x_i - (x_0)_i| + |x-y| < |x_i - (x_0)_i| + \e = |x_i - (x_0)_i| + \frac{\eta \rho}{2} 
$$
 for some index $i\in\{1,\ldots,d\}$;
that is, $x \in Q_\e \setminus Q_{\rho(1-2\eta)}(x_0) \subset Q_\rho(x_0) \setminus Q_{\rho(1-2\eta)}(x_0) $.  By Proposition \ref{lemma applied to u_eps}, it follows that for $(x,y) \in  (Q_\e \times (Q_{\rho(1+ \eta)}(x_0) \setminus Q_\e)) \cap \{|x-y|<\eps \}  $, we have simultaneously
$$
u_\e(x)=u(x_0) + \nabla u(x_0)\cdot (x-x_0),
$$
$$
u_\e(y)=u(x_0) + \nabla u(x_0)\cdot (y-x_0),
$$
(the latter by periodicity in $y$). Therefore, by (\ref{splitting of Q eps x Q eps}), we can estimate
\begin{align*}
    \frac{1}{\rho^d\eps^{d+2}} &\int_{  \substack{(Q_\e \times \R^d) \cap (\delta E)^2 \\ \{|x-y|<\eps\}}}(u_{\eps}(x)-u_{\eps}(y))^2dxdy
    \leq
    \frac{1}{\rho^d\eps^{d+2}} \int_{  \substack{(Q_{\e} \times Q_\eps) \cap (\delta E)^2 \\ \{|x-y|<\eps\}}}(u_{\eps}(x)-u_{\eps}(y))^2dxdy \\ &+
    \frac{1}{\rho^d\eps^{d+2}} \int_{  \substack{(Q_\e \times (Q_{\rho(1+ \eta)}(x_0) \setminus Q_\e)) \cap (\delta E)^2 \\ \{|x-y|<\eps\}}}(\nabla u (x_0) \cdot (x-y))^2dxdy \\ 
    &\leq \frac{1}{\rho^d\eps^{d+2}}\int_{  \substack{(Q_{\e }\cap \delta E)^2 \\ \{|x-y|<\eps\}}}(u_{\eps}(x)-u_{\eps}(y))^2dxdy +\frac{|\nabla u(x_0)|^2 \eps^2}{\rho^d \e^{d+2}} \int_{Q_{\rho(1+ \eta)}(x_0) \setminus Q_\e } \left(\int_{Q_\e \cap B(y,\eps)} dx \right) dy
    \\ 
    &\le
   \frac{1}{\rho^d\eps^{d+2}} \int_{  \substack{(Q_{\e }\cap \delta E)^2 \\ \{|x-y|<\eps\}}}(u_{\eps}(x)-u_{\eps}(y))^2dxdy + \frac{|\nabla u(x_0)|^2 \eps^2}{\rho^d \e^{d+2}} \, |Q_{\rho(1+\eta)} \setminus Q_{\rho(1-\eta)} | \e^d \\
    &\le \frac{1}{\rho^d\eps^{d+2}} \int_{  \substack{(Q_{\e }\cap \delta E)^2 \\ \{|x-y|<\eps\}}}(u_{\eps}(x)-u_{\eps}(y))^2dxdy + \frac{|\nabla u(x_0)|^2}{\rho^d }4d \eps \rho^{d-1} \\
   & = \frac{1}{\rho^d\eps^{d+2}} \int_{  \substack{(Q_{\e }\cap \delta E)^2 \\ \{|x-y|<\eps\}}}(u_{\eps}(x)-u_{\eps}(y))^2dxdy + C \frac{\e}{\rho}
\end{align*}
since $|Q_{\rho(1+\eta)} \setminus Q_\e  | \leq |Q_{\rho(1+\eta)} \setminus Q_{\rho(1-\eta)}|= [\rho(1+\eta)]^d - [\rho(1-\eta)]^d  \leq 2 (\eta\rho) d \rho^{d-1}=4d\eps \rho^{d-1}$. We can read the estimate above as
\begin{multline*}
  \frac{1}{\rho^d\eps^{d+2}} \int_{  \substack{(Q_{\e }\cap \delta E)^2 \\ \{|x-y|<\eps\}}}(u_{\eps}(x)-u_{\eps}(y))^2dxdy \geq
   \frac{1}{\rho^d\eps^{d+2}} \int_{  \substack{(Q_\e \times \R^d) \cap (\delta E)^2 \\ \{|x-y|<\eps\}}}(u_{\eps}(x)-u_{\eps}(y))^2dxdy - C \frac{\e}{\rho} \end{multline*}
and $C\frac{\e}{\rho}=o(1)$ as $\e \to 0$.

Set $w_\e(z) := u_\e(x_0 + l_\e z)-\nabla u (x_0) \cdot l_\e z$ for all $z \in Q$, which is $\mathbb{Z}^d$-periodic by our assumption. For simplicity denote $M_\e=1+2N_\e$, so that $l_\e=(1+2N_\e)\delta=M_\e \delta$ and $|Q_\e|^d=l_\e^d=M_\e^d\delta^d$. We have
\begin{eqnarray*}
&& \frac{\mu_\e(Q_{\rho}(x_0)) }{\rho^d}= \frac{1}{\rho^d\eps^{d+2}} \int_{  \substack{(Q_{\rho}(x_0)\times \Omega) \cap (\delta E)^2 \\ \{|x-y|<\eps\}}}(u_{\eps}(x)-u_{\eps}(y))^2dxdy \\ 
&\ge&
  \frac{1}{\rho^d\eps^{d+2}} \int_{  \substack{(Q_\e\cap \delta E)^2 \\ \{|x-y|<\eps\}}}(u_{\eps}(x)-u_{\eps}(y))^2dxdy\\&  \ge&
    \frac{1}{\rho^d\eps^{d+2}} \int_{  \substack{(Q_{\e}\times \R^d)\cap (\delta E)^2 \\ \{|x-y|<\eps\}}}(u_{\eps}(x)-u_{\eps}(y))^2dxdy +o(1) 
    \\ &=&
  \frac{1}{\rho^d\eps^{d+2}} \int_{  \substack{(Q_{\e} \times \R^d)\cap (\delta E)^2 \\ \{|x-y|<\eps\}}}\Big(w_\e\Big (\frac{x-x_0}{l_\e }\Big)-w_\e\Big (\frac{y-x_0}{l_\e }\Big) + \nabla u(x_0)\cdot(x- y )\Big)^2dxdy +o(1)  \\ &\ge&
 \frac{1}{\rho^d\eps^{d+2}}  \inf_{w \, \Z^d\text{-per.}} \bigg\{\int_{  \substack{(Q_\e \times \R^d)\cap (\delta E)^2 \\ \{|x-y|<\eps\}}}\Big(w\Big (\frac{x-x_0}{l_\e }\Big)-w\Big (\frac{y-x_0}{l_\e }\Big) + \nabla u(x_0)\cdot(x- y )\Big)^2dxdy \bigg\} \label{last line mu eps}.
\end{eqnarray*}
We work on the infimum by change of variables, writing
\begin{eqnarray*}
&& \inf_{w \, \Z^d\text{-per.}} \int_{  \substack{(Q_\e \times \R^d)\cap (\delta E)^2 \\ \{|x-y|<\eps\}}}\Big(w\Big (\frac{x-x_0}{l_\e }\Big)-w\Big (\frac{y-x_0}{l_\e }\Big) + \nabla u(x_0)\cdot(x- y )\Big)^2dxdy   \\
&=&  l_\e^{2d} \inf_{w \, \Z^d\text{-per.}} \int_{  \substack{\left(Q_1\left(\frac{x_0}{l_\e} \right) \times \R^d \right)\cap \left( \frac{E}{M_\e}\right)^2 \\ \{|x-y|<\frac{\eps}{l_\e}\}}}\Big(w\Big (x-\frac{x_0}{l_\e }\Big)-w\Big (y-\frac{x_0}{l_\e }\Big)  + \nabla u(x_0)\cdot l_\e(x- y )\Big)^2dxdy \\
&=& l_\e^{2d+2} \inf_{w \, \Z^d\text{-per.}} \int_{  \substack{\left(Q_1\left(\frac{x_0}{l_\e} \right) \times \R^d \right)\cap \left( \frac{E}{M_\e}\right)^2 \\ \{|x-y|<\frac{\eps}{l_\e}\}}}\left(w\left(x\right)-w\left (y\right)  + \nabla u(x_0)\cdot (x- y )\right)^2dxdy.
\end{eqnarray*}
The last equality comes from the fact that, for each $w$ 1-periodic, the transformation $\tilde w(x)= l_\eps w\big(x+\frac{x_0}{l_\e}\big)$ (which is also $1$-periodic) gives the same integral, and the infimum is the same by the one-to-one correspondence. Now we use the periodicity of $E$ to get rid of the center $\frac{x_0}{l_\e}$, indeed it holds
\begin{eqnarray*}&&
\int_{  \substack{\left(Q_1\left(\frac{x_0}{l_\e} \right) \times \R^d \right)\cap \left( \frac{E}{M_\e}\right)^2 \\ \{|x-y|<\frac{\eps}{l_\e}\}}}\left(w\left(x\right)-w\left (y\right)  + \nabla u(x_0)\cdot (x- y )\right)^2dxdy \\ &=&
\int_{  \substack{\left(Q_1\left(\frac{x_0}{l_\e} \right) \times \R^d \right) \\ \{|x-y|<\frac{\eps}{l_\e}\}}}\chi_{ \left(\frac{E}{M_\e} \times \frac{E}{M_\e}\right) }(x,y)\left(w\left(x\right)-w\left (y\right)  + \nabla u(x_0)\cdot (x- y )\right)^2dxdy \\
&=&\int_{  \substack{\left(Q_1\left(0 \right) \times \R^d \right) \\ \{|x-y|<\frac{\eps}{l_\e}\}}}\chi_{ \left(\frac{E}{M_\e} \times \frac{E}{M_\e}\right) }(x,y)\left(w\left(x\right)-w\left (y\right)  + \nabla u(x_0)\cdot (x- y )\right)^2dxdy,
\end{eqnarray*}
given that the integrand is jointly $1$-periodic in both $x$ and $y$, thanks to the $1$-periodicity of $E$ (and in particular of $\frac{E}{M_\e}$), of $w$ and that $y$ runs in $\R^d$ (so that, when $x$ and $y$ is changed with $x+k$ and $y+k'$, the change of variable $y'=y+k'-k$ works). 

Let $w$ be $\Z^d$-periodic and define a $\frac{1}{M_\e}\Z^d$-periodic function
\begin{align*}
    \tilde {w}(x)  :=  \frac{1}{M_\e^d}\sum_{k \in \{ -N_\e,...,0,...,N_\e\}^d } w\Big(x+ \frac{k}{M_\e}\Big)
\end{align*}
Via a standard convexity argument, calling $\mathcal{I}_\eps=\{ -N_\e,...,0,...,N_\e\}^d$,
\begin{eqnarray*}&&
\int_{  \substack{\big(Q_\frac{1}{M_\e} \times \R^d \big) \cap \left(\frac{E}{M_\e} \right)^2 \\ \{|x-y|<\frac{\eps}{l_\e}\}}} \left(\tilde w(x)-\tilde w(y)  + \nabla u(x_0)\cdot (x- y )\right)^2dxdy \\
&=& \frac{1}{M_\e^d} \int_{  \substack{\left(Q_1 \times \R^d \right) \cap \left(\frac{E}{M_\e} \right)^2 \\ \{|x-y|<\frac{\eps}{l_\e}\}}} \left(\tilde w(x)-\tilde w(y)  + \nabla u(x_0)\cdot (x- y )\right)^2dxdy \\
&=& \frac{1}{M_\e^d} \int_{  \substack{\left(Q_1 \times \R^d \right) \cap \left(\frac{E}{M_\e} \right)^2 \\ \{|x-y|<\frac{\eps}{l_\e}\}}}
\Big( \frac{1}{M_\e^d} \sum_{k \in \mathcal{I}_\eps} \Big( w\Big(x+ \frac{k}{M_\e}\Big) - w\Big(y+ \frac{k}{M_\e}\Big) + \nabla u(x_0)\cdot(x-y) \Big)  \Big)^2dxdy \\
&\leq& \frac{1}{M_\e^{2d}} \sum_{k \in \mathcal{I}_\eps}\int_{  \substack{\left(Q_1 \times \R^d \right) \cap \left(\frac{E}{M_\e} \right)^2 \\ \{|x-y|<\frac{\eps}{l_\e}\}}} \Big( w\Big(x+ \frac{k}{M_\e}\Big) - w\Big(y+ \frac{k}{M_\e}\Big) + \nabla u(x_0)\cdot(x-y) \Big)^2dxdy \\&=& \frac{1}{M_\e^{d}} \int_{  \substack{\left(Q_1 \times \R^d \right) \cap \left(\frac{E}{M_\e} \right)^2 \\ \{|x-y|<\frac{\eps}{l_\e}\}}} \left( w\left(x\right) - w\left(y\right) + \nabla u(x_0)\cdot(x-y) \right)^2dxdy.
\end{eqnarray*}
Thus by taking the infimum for every $\tilde w$ $\frac{1}{M_\e}\Z^d$-periodic on the left-hand side of the last inequality, by the arbitrariness of $w$ $1$-periodic, in (\ref{last line mu eps}) we get 
\begin{eqnarray*}
&&\frac{\mu_\e(Q_\rho(x_0))}{\rho^d}\\
&\geq& \frac{l_\e^{2d+2}}{\rho^d \e^{d+2}}M_\e^{d} \inf_{w \, \frac{1}{M_\e}\Z^d\text{-per.}} \int_{  \substack{\left(Q_\frac{1}{M_\e} \times \R^d \right) \cap \left(\frac{E}{M_\e} \right)^2 \\ \{|x-y|<\frac{\eps}{l_\e}\}}} \left( w(x)-w(y)  + \nabla u(x_0)\cdot (x- y )\right)^2dxdy \\
&=&\underbrace{\frac{l_\e^d}{\rho^d}}_{\sim 1 \text{ as } \eps \to 0}\frac{\delta^{d+2}}{\eps^{d+2}}\inf_{w \, \frac{1}{M_\e}\Z^d\text{-per.}} \int_{  \substack{\left(Q_\frac{1}{M_\e} \times \R^d \right) \cap \left(\frac{E}{M_\e} \right)^2 \\ \{|x-y|<\frac{\eps}{l_\e}\}}}M_\e^{2d+2} \left( w(x)-w(y)  + \nabla u(x_0)\cdot (x- y )\right)^2dxdy \\
&=&\frac{\delta^{d+2}}{\eps^{d+2}}\inf_{w \, \frac{1}{M_\e}\Z^d\text{-per.}} \int_{  \substack{\left(Q_1 \times \R^d \right) \cap (E)^2 \\ \{|x-y|<\frac{\eps}{\delta}\}}} \left( w\left(\frac{x}{M_\e}\right)-w\left(\frac{y}{M_\e}\right)  + \nabla u(x_0)\cdot (x- y )\right)^2dxdy 
\\ &\geq& \frac{\delta^{d+2}}{\eps^{d+2}}\inf_{w \, \Z^d\text{-per.}} \int_{  \substack{\left(Q_1 \times \R^d \right) \cap (E)^2 \\ \{|x-y|<\frac{\eps}{\delta}\}}} \left( w(x)-w(y)  + \nabla u(x_0)\cdot (x- y )\right)^2dxdy, 
\end{eqnarray*}
where the second equality holds from a change of variable and the same value of the inf with $M_\e w$, and the last inequality comes from the fact that $x \mapsto w\big( \frac{x}{M_\e} \big)$ is $1$ periodic if $w$ is $\frac{1}{M_\e}$ periodic, and the infimum decreases. Finally, with fixed $\eta >0$, for $\e$ small enough we have 
$\eta >\frac{\delta}{\eps} \text{diam}(K)$,
and we set $Z_\e=\{ k \in \Z^d: \frac{\delta}{\eps}|k| < 1-\eta\} $. Note that, when $k \in Z_\e$  we have 
$$
|x-y-k| \leq \text{diam}(K)+|k| < \frac{\eps}{\delta} - \frac{\eps}{\delta}\eta +\text{diam(K)} < \frac{\eps}{\delta}
$$
so that $(x,y) \in (K \times E) \cap \{|x-y|< \frac{\eps}{\delta} \}=(K \times (K+k)) \cap \{|x-y|< \frac{\eps}{\delta} \}$. By a double Jensen's inequality we have 
\begin{align*}
 &\inf_{w \, \Z^d\text{-per.}} \int_{  \substack{(Q_1 \times \R^d)\cap (E)^2 \\ \{|x-y|< \frac{\eps}{\delta}\}}}\left(w(x)-w(y) \right.
+ \nabla u(x_0)\cdot(x- y ))^2dxdy\\&=
\inf_{w \, \Z^d\text{-per.}} \int_{ (K \times E) \cap \{ |x-y|<\frac{\eps}{\delta}\}}\left(w\left (x\right)-w\left (y\right)
+ \nabla u(x_0)\cdot(x- y )\right)^2dxdy \\ &\ge
\inf_{w \, \Z^d\text{-per.}} \sum_{k \in Z_\eps}\int_{ K \times (K+k) }\left(w\left (x\right)-w\left (y\right)
+ \nabla u(x_0)\cdot(x- y )\right)^2dxdy \\ &\ge
\inf_{w \, \Z^d\text{-per.}} \sum_{k \in Z_\eps}|K|^{2}\left(\frac{1}{|K|^2}\int_{ K \times (K+k) }w\left (x\right)-w\left (y\right)
+ \nabla u(x_0)\cdot(x- y )dxdy\right)^2\\ &=
\sum_{k \in Z_\eps}|K|^{2}\left(\frac{1}{|K|^2}\int_{ K \times K }\nabla u(x_0)\cdot k \, dxdy\right)^2 =
|K^2| \sum_{k \in Z_\eps} \left( \nabla u(x_0)\cdot k\right)^2,
\end{align*}
which is independent of $w$ $1$-periodic. Thus
\begin{align*}
\frac{\mu_\e(Q_\rho(x_0))}{\rho^d} \geq \frac{\delta^{d+2}}{\eps^{d+2}}|K^2| \sum_{k \in Z_\eps} \left( \nabla u(x_0)\cdot k\right)^2    
\end{align*}
and
\begin{align*}
\liminf_{\eps \to 0}  \frac{\mu_\e(Q_{\rho}(x_0)) }{\rho^d} \geq |K^2| \liminf_{\eps \to 0} \sum_{k \in Z_\eps}\left(\frac{\delta}{\eps}\right)^d  \left( \nabla u(x_0)\cdot \frac{\delta}{\eps}k\right)^2=   |K|^2\int_{B_{1-\eta}} |\nabla u(x_0)\cdot \xi|^2 d\xi.
\end{align*}
%and let $\eta \to 0$. 
Using the lower semicontinuity of the mass,
    $$
    \liminf_{\eps \to 0} F_{\eps}(u_{\eps}) \geq \mu(\Omega) \geq \int_{\Omega} \frac{d \mu}{dx} dx,
    $$    
we get
    $$
    \liminf_{\eps \to 0} F_{\eps}(u_{\eps}) \geq |K|^2 \int_{\Omega}\int_{B_1} |\nabla u(x) \cdot \xi|^2 d\xi dx= |K|^2 C(d)\int_{\Omega} |\nabla u (x)|^2 dx,
$$
with $C_d := \frac{1}{d}\int_{B_1}|\xi|^2 d\xi$ as in \cite{bourgainbrezis} (for us $\phi(z)$ was the characteristic function of the unit ball).

\subsection{Limsup inequality}\label{subsection limsup inequality}
Let $u \in C^2(\overline \Omega)$ and let $\eta > 0$ be fixed. Let $\Omega \Subset \Omega'$ and set $Z_\delta  := \left\{k \in \mathbb{Z}^d : Q_\delta^k\cap \Omega \neq \emptyset\right\}$.
For $\delta \ll 1$ we have $\bigcup_{k \in Z_{\delta}}Q_{\delta}^k \subseteq \{x \in \R^d: \text{dist}(x, \Omega) \le 2\delta\} \subseteq \Omega'$. Moreover, without loss of generality we can assume $u \in C^2(\Omega')$.

By the regularity assumption on $u$ and a weighted Cauchy-Schwarz inequality:
\begin{align*}
    |u(x)-u(y)|^2&=|\nabla u (x) \cdot (x-y) + R(x,y)|^2 \\
    &\le (1+\eta) |\nabla u (x)\cdot (x-y)|^2 + \Big(1+\frac{1}{\eta}\Big)|R(x,y)|^2, 
\end{align*}
with $|R(x,y)|\le C|x-y|^2.$ 
Moreover, in the assumption $\eps \gg \delta$, it also holds
$$
|k\delta -k' \delta| \le |k\delta -x| + |k'\delta -y| + \eps  \le \eps + 2\delta \le \eps(1+\eta)
$$
for all $k, k' \in Z_{\delta}$ and $x, y \in  Q_{\delta}^k \times Q_{\delta}^{k'} $ such that $|x-y|< \eps$, which implies
\begin{eqnarray*}
F_\e(u) &\le& (1+\eta) \frac{1}{\e^{d+2}} \int_{\substack{(\Omega \cap \delta E)^2 \\ \{|x-y|<\e\}} }|\nabla u (x) \cdot (x-y)|^2 dx dy + C_\eta \frac{1}{\e^{d+2}} \int_{\substack{(\Omega \cap \delta E)^2 \\ \{|x-y|<\e\}}} |x-y|^4 dxdy\\
&\le& (1+\eta)\frac{1}{\eps^{d+2}}\sum_{k,k' \in Z_{\delta}} \int_{\substack{(Q_{\delta}^k \times Q_{\delta}^{k'}) \cap (\delta E)^2 \\ \{|x-y|<\e\}}}
 |\nabla u (x) \cdot (x-y)|^2 dx dy + C_\eta\e^2 \\
&\le& (1+\eta)\frac{1}{\eps^{d+2}}\sum_{\substack{k,k' \in Z_{\delta} \\ \delta|k-k'| < (1+\eta) \e} } \int_{(Q_{\delta}^k \times Q_{\delta}^{k'}) \cap (\delta E)^2 }
 |\nabla u (x) \cdot (x-y)|^2 dx dy + C_\eta \e^2.
\end{eqnarray*}
Note that for all $k,k' \in Z_{\delta}$ and for all $x= \delta k + z$ and $y=\delta k' + w$, with $z,w \in \delta Q$,
 \begin{eqnarray*}
 |\nabla u(x) \cdot (x-y)| &=& | \nabla u(x)\cdot (\delta k -\delta k') + \nabla u(x) \cdot (z-w)|\\
  &\leq& | \nabla u(x)\cdot (\delta k -\delta k')| + \|\nabla u \|_{\infty} \delta 
 \\ &\leq&  | \nabla u(k\delta )\cdot (\delta k -\delta k')| +\|D^2 u \|_{\infty}  \delta +\|\nabla u \|_{\infty} \delta\\
 & \leq& | \nabla u(k\delta )\cdot (\delta k -\delta k')| +C \delta, 
 \end{eqnarray*}
 so that $ |\nabla u(x) \cdot (x-y)|^2 \leq (1+\eta)| \nabla u(k\delta )\cdot (\delta k -\delta k')|^2 + C_\eta \delta^2$. 
 Therefore we have
\begin{eqnarray*}
 &&  \frac{1}{\eps^{d+2}}\sum_{\substack{k,k' \in Z_{\delta} \\ \delta|k-k'| < (1+\eta) \e} } \int_{(Q_{\delta}^k \times Q_{\delta}^{k'}) \cap (\delta E)^2 }
 |\nabla u (x) \cdot (x-y)|^2 dx dy \\
 &\le& (1+\eta) \frac{1}{\eps^{d+2}} \sum_{\substack{k,k' \in Z_{\delta} \\ \delta|k-k'| < (1+\eta) \e} } \int_{(Q_{\delta}^k \times Q_{\delta}^{k'}) \cap (\delta E)^2 }
 |\nabla u (k\delta) \cdot (k\delta-k'\delta)|^2 dx dy  
\\
 &&+ C_\eta \sum_{\substack{k,k' \in Z_{\delta} \\ \delta|k-k'| < (1+\eta) \e} }  \frac{1}{\eps^{d+2}}\delta^{2d}\delta^2 \\
& =&(1+\eta) \frac{1}{\eps^{d+2}} \sum_{\substack{k,k' \in Z_{\delta} \\ \delta|k-k'| < (1+\eta) \e} } \int_{(Q_{\delta}^k \times Q_{\delta}^{k'}) \cap (\delta E)^2 }
 |\nabla u (k\delta) \cdot (k\delta-k'\delta)|^2 dx dy  
 \\
&& + C_\eta \sum_{k \in Z_{\delta}} \sum_{|j|< \frac{(1+\eta)\e}{\delta}} \frac{1}{\eps^{d+2}}\delta^{2d}\delta^2 
 \\
 &\le& (1+\eta)\frac{1}{\eps^{d+2}} \sum_{\substack{k,k' \in Z_{\delta} \\ \delta|k-k'| < (1+\eta) \e} } \int_{(Q_{\delta}^k \times Q_{\delta}^{k'}) \cap (\delta E)^2 }
 |\nabla u (k\delta) \cdot (k\delta-k'\delta)|^2 dx dy 
 \\
 &&+C_\eta \frac{1}{\delta^d}\left(\frac{\e}{\delta}\right)^d \frac{1}{\eps^{d+2}}\delta^{2d}\delta^2 
 \\
 &\le& (1+\eta) \frac{1}{\eps^{d+2}} \sum_{\substack{k,k' \in Z_{\delta} \\ \delta|k-k'| < (1+\eta) \e} } |(Q_{\delta}^k \times Q_{\delta}^{k'}) \cap (\delta E)^2|
 |\nabla u (k\delta) \cdot (k\delta-k'\delta)|^2  +C_\eta \left(\frac{\delta}{\eps}\right)^2.
\end{eqnarray*}
Finally, by the periodicity of the set $E$
\begin{eqnarray*}  
|(Q_{\delta}^k \times Q_{\delta}^{k'}) \cap (\delta E)^2|&=& |Q_{\delta}^k \cap \delta E||Q_{\delta}^{k'} \cap \delta E|\\
&=& \delta^{2d} |(k+Q) \cap E| |(k'+Q) \cap E| 
= \delta^{2d} |Q\cap E|^2= \delta^{2d} |K|^2 ,
\end{eqnarray*}
which implies 
\begin{eqnarray*}
 &&   \frac{1}{\e^{d+2}}\sum_{\substack{k,k' \in Z_{\delta} \\ \delta|k-k'| < (1+\eta) \e} } |(Q_{\delta}^k \times Q_{\delta}^{k'}) \cap (\delta E)^2|
 |\nabla u (k\delta) \cdot (k\delta-k'\delta)|^2  \\
 &=&\frac{1}{\e^{d+2}}\sum_{\substack{k,k' \in Z_{\delta} \\ \delta|k-k'| < (1+\eta) \e} } |K|^2 \delta^{2d}
 |\nabla u (k\delta) \cdot (k\delta-k'\delta)|^2  \\
 &\le& |K|^2 \sum_{k \in Z_\delta} \delta ^ d \sum_{|j|< \frac{(1+\eta)\e}{\delta}} \Big(\frac{\delta}{\e}\Big)^d\Big|\nabla u (k\delta)\cdot \frac{\delta}{\e}j\Big|^
 2.
\end{eqnarray*}
It follows
\begin{align*}
    F_\e(u) \le (1+\eta)^2|K|^2 \sum_{k \in Z_\delta} \delta ^ d \sum_{\frac{\delta}{\eps}|j|< 1+\eta} \Big(\frac{\delta}{\e}\Big)^d\Big|\nabla u (k\delta)\cdot \frac{\delta}{\e}j\Big|^
 2 + C_\eta \Big(\Big(\frac{\delta}{\eps}\Big)^2+ \e^2\Big),
\end{align*}
and, thus, 
$$
\limsup_{\e\to 0} F_\e (u) \le (1+\eta)^2|K|^2\int_{\Omega' \times B_{1+2\eta}} |\nabla u (x) \cdot \xi|^2 dx d\xi.
$$
Sending $\eta \to 0$ and $\Omega ' \downarrow \Omega$ we conclude
$$
\limsup_{\e\to 0} F_\e (u) \le |K|^2\int_{\Omega \times B_1}|\nabla u (x) \cdot \xi|^2 dx d\xi = |K|^2C(d)\int_{\Omega} |\nabla u (x)|^2 dx,
$$
with $C_d = \frac{1}{d}\int_{B_1}|\xi|^2 d\xi$.

\subsection{$\Gamma$-convergence with general convolution kernel $\phi_{\eps}$}
We recall that the general form of the functional we consider is
    \begin{equation}\label{functional general phi section}
    F_{\eps}(u)=\frac{1}{\eps^{d+2}} \int_{(\Omega \cap \delta E) \times (\Omega \cap \delta E)} \phi\Big(\frac{x-y}{\eps}\Big)(u(x)-u(y))^2dxdy
    \end{equation}
and our preliminary case considered is $\phi(\xi)=\chi_{(0,1)}(|\xi|)$.    
In the case of a general $\phi$ we can proceed by approximation. As stated at the beginning of Section \ref{section: notation and statement of the results}, we consider a radial convolution kernel in $\mathbb R^d$; that is, a function $\varphi:\mathbb R^d\to [0,+\infty)$ such that a decreasing function $\phi_0:[0,+\infty)\to  [0,+\infty)$ exists satisfying $\varphi(\xi)=\phi_0(|\xi|)$. We further assume that
\begin{equation*}
\int_{\mathbb R^d} \varphi(\xi)(1+|\xi|^2)dx<+\infty,
\end{equation*}
and for each $\e>0$ we define the scaled kernel $\varphi_\e$ by $\varphi_\e(\xi)= \tfrac1{\e^d}\varphi(\tfrac\xi\e)$.
 
Since $\phi_0$ is decreasing, we can write a `staircase' approximation made of characteristic functions of nested intervals $(0,r_k)$ (where $r_k$ are defined below), that is 
\begin{equation}\label{staircase approximation}
\phi_{0,n}(t)=\sum_{k=0}^{n2^n} \chi_{(0,r_k)}(t)
\end{equation}
where for each $k=0,\dots, n2^n$
    $$
    r_k= \sup \left\{ t : \phi_0(t) \geq (k+1)2^{-n}\right\}
    $$
and $r_k=0$ if the set of such $t$ is empty. With this choice, $r_j$ is decreasing by the monotonicity of $\phi_0$. Accordingly, for each $n \in \mathbb{N}$ define the simple function
  $$
  \phi_n(\xi)= \sum_{k=0}^{n2^n} \chi_{(0,r_k)}(|\xi|).
  $$
Note that $\phi_n(\xi) \leq \phi(\xi)$ for every $\xi \in \R^d$. By standard approximation through simple functions, $\phi_n(\xi) \to \phi(\xi)$ pointwise almost everywhere, in $L^1(\R^d)$ by dominated convergence and, since $\phi$ is in particular bounded, it converges also uniformly. 

We define as $F^n_{\eps}$ the same functional (\ref{functional general phi section}) with $\phi_n$ in place of $\phi$. 
Since $\phi$ is an upper bound for any $\phi_n$, we have for any $u_\e \rightharpoonup u$ and $n \in \mathbb{N}$ fixed
    \begin{align*}
    \liminf_{\eps \to 0} F_{\eps}(u_{\eps})  \geq \liminf_{\eps \to 0} F^n_{\eps}(u_{\eps})&= \liminf_{\eps \to 0} \frac{1}{\eps^{d+2}} \int_{(\Omega \cap \delta E)^2 } \phi_n\left(\frac{x-y}{\eps}\right)(u_\eps(x)-u_\eps(y))^2dxdy \\
    &=\liminf_{\eps \to 0}  \Big( \sum_{k=0}^{n2^n} \frac{1}{\eps^{d+2}} \int_{ \substack{ (\Omega \cap \delta E) ^2 \\ \{|x-y|<r_k \eps\}}}(u_\eps(x)-u_\eps(y))^2dxdy \Big) \\
    &\geq \sum_{k=0}^{n2^n} \liminf_{\eps \to 0}  \frac{1}{\eps^{d+2}} \int_{ \substack{ (\Omega \cap \delta E)^2 \\ \{|x-y|<r_k \eps\}}}(u_\eps(x)-u_\eps(y))^2dxdy \\
    &\geq \sum_{k=0}^{n2^n}  |K|^2 \int_{\Omega}\int_{B_{r_k}} |\nabla u(x) \cdot \xi|^2 d \xi dx= \\
    &=|K|^2 \int_{\Omega} \int_{\R^d} \phi_n(\xi) |\nabla u(x) \cdot \xi|^2d \xi dx, 
    \end{align*} 
by the result of Section \ref{subsection liminf inequality}. We can pass to the limit for $n \to \infty$ by monotone convergence and we get
$$
 \liminf_{\eps \to 0} F_{\eps}(u_{\eps})  \geq |K|^2 \int_{\Omega} \int_{\R^d} \phi(\xi) |\nabla u(x) \cdot \xi|^2d \xi dx= C_\infty \int_{\Omega} |\nabla u(x)|^2dx,
$$
with $C_\infty= |K|^2 \frac{1}{d} \int_{\R^d} \phi(\xi) |\xi|^2 d \xi $. Hence, the lower bound is satisfied. 

\smallskip
As for the upper bound, we claim that it is enough to consider the function $\phi$ with compact support. Indeed, for any $R>0$ we can split the functional 
\begin{eqnarray*}
F_\e(u)&=&\frac{1}{\e^{d+2}} \int_{(\Omega \cap \delta E)^2} \phi\Big(\frac{x-y}{\eps}\Big)(u(x)-u(y))^2dxdy \\
&=&\frac{1}{\e^{d+2}} \int_{(\Omega \cap \delta E)^2} \phi\Big(\frac{x-y}{\eps}\Big)\,\chi_{B_R}\Big(\frac{x-y}{\e}\Big)(u(x)-u(y))^2dxdy  \\
&&+ \frac{1}{\e^{d+2}} \int_{(\Omega \cap \delta E)^2} \phi\Big(\frac{x-y}{\eps}\Big) \,\chi_{\R^d \setminus B_R}\Big(\frac{x-y}{\e}\Big)(u(x)-u(y))^2dxdy.
\end{eqnarray*}
In the second term of the sum we can perform the change of variables $y=x+\e \xi$, so that, setting $\Omega(\xi)= \left\{ x \in \Omega \cap \delta E: x+\eps \xi \in \Omega \cap \delta E \right\} \subset \Omega $ for each $\xi \in \R^d$, it holds
\begin{eqnarray*}
&&\hskip-2cm\frac{1}{\e^{d+2}} \int_{(\Omega \cap \delta E)^2} \phi\Big(\frac{x-y}{\eps}\Big) \chi_{\R^d \setminus B_R}\Big(\frac{x-y}{\e}\Big)(u(x)-u(y))^2dxdy 
\\ 
&=& \int_{\R^d \setminus B_R} \phi(\xi) \int_{\Omega(\xi)} \bigg( \frac{u(x+ \eps \xi)-u(x)}{\eps}\bigg)^2 dx d\xi 
\\
&\leq&  \int_{\R^d \setminus B_R} \phi(\xi)  \int_{\Omega } \left(\| \nabla u \|_{L^\infty(\Omega)} |\xi| \right)^2   dx d\xi \\
&=& |\Omega|  \| \nabla u \|_{L^\infty(\Omega)} \int_{\R^d \setminus B_R} \phi(\xi)|\xi|^2 d\xi=o(1)
\end{eqnarray*}
as $R \to \infty$, uniformly in $\e$. Thus we are left with the first term where $\phi \chi_{B_R}$ has compact support, so that the claim holds. 
For a general $\phi_0$ bounded, decreasing and with compact support, we can approximate it from above in the same way of ($\ref{staircase approximation}$), since for $\phi_{0,n}$ it holds
$$
\phi_{0,n} \leq \phi_0 \leq \phi_{0,n} +2^{-n} 
$$
for each $t \in \R$, so it it enough to take $\tilde \phi_{0,n}(t)= \phi_{0,n}(t)+2^{-n}$ for $t \in \text{supp}(\phi_0)=(0,R)$, and $0$ otherwise. Accordingly for each $n \in \mathbb{N}$ define 
$$
\tilde \phi_n (\xi)= \tilde \phi_{0,n}(|\xi|)=\sum_{k=0}^{n2^n} \chi_{(0,r_k)}(|\xi|)+ 2^{-n} \chi_{(0,R)}(|\xi|)= \sum_{k \in \mathcal{I}_n} \chi_{(0,r_k)}(|\xi|)     
$$
($\mathcal{I}_n$ is just a finite set of indices for each $n$) for which $\tilde \phi_n \geq \phi$ and it converges to $\phi$ uniformly and in $L^1(\R^d)$. Finally, for $u \in C^2(\overline{\Omega})$ and $n \in \mathbb{N}$ fixed, by the result of Section \ref{subsection limsup inequality}
\begin{align*}
    \limsup_{\eps \to 0} F_{\eps}(u)  &\leq  \limsup_{\eps \to 0} \frac{1}{\eps^{d+2}} \int_{(\Omega \cap \delta E)^2 } \tilde \phi_n\Big(\frac{x-y}{\eps}\Big)(u(x)-u(y))^2dxdy \\
    &=\limsup_{\eps \to 0}  \bigg( \sum_{k \in \mathcal{I}_n} \frac{1}{\eps^{d+2}} \int_{ \substack{ (\Omega \cap \delta E)^2 \\ \{|x-y|<r_k \eps\}}}(u(x)-u(y))^2dxdy \bigg) \\
    &\leq \sum_{k \in \mathcal{I}_n} \lim_{\eps \to 0}  \frac{1}{\eps^{d+2}} \int_{ \substack{ (\Omega \cap \delta E)^2 \\ \{|x-y|<r_k \eps\}}}(u(x)-u(y))^2dxdy \\
    &= \sum_{k \in \mathcal{I}_n}  |K|^2 \int_{\Omega}\int_{B_{r_k}} |\nabla u(x) \cdot \xi|^2 d\xi dx= \\
    &=|K|^2 \int_{\Omega} \int_{\R^d} \tilde \phi_n(\xi) |\nabla u(x) \cdot \xi|^2d \xi dx ,
    \end{align*} 
and finally pass to the limit as $n \to \infty$ to get the upper bound
$$
\limsup_{\eps \to 0} F_{\eps}(u) \leq |K|^2 \int_{\Omega} \int_{\R^d} \phi(\xi) |\nabla u(x) \cdot \xi|^2d \xi dx=C_\infty \int_{\Omega} |\nabla u(x)|^2 dx. 
$$

\noindent {\bf Acknowledgements.}
The research reported in the present contribution was carried out as part of the project ``Variational methods for stationary and evolution problems with singularities and interfaces" PRIN 2022J4FYNJ. 
The authors are members of GNAMPA of INdAM.

%\newpage \ \\ Citazioni per provare la bibliografia \cite{alicandro2023variational}, \cite{solci2023nonlocalinteraction}, \cite{Maggi_2012}, \cite{bourgainbrezis} .
 \bibliographystyle{abbrv}
\bibliography{bibliography}

%\printbibliography

\end{document}